\documentclass[11pt]{amsart}
\usepackage{latexsym,amssymb,amsmath,hyperref,amsthm,amsfonts,
caption,subcaption,tikz,comment}
\usepackage{mdwlist,dirtytalk,times}
\usepackage[capitalise, noabbrev]{cleveref}
\usetikzlibrary{
intersections, arrows.meta,
automata,er,calc,
backgrounds,
mindmap,folding,
patterns,
decorations.markings,
fit,decorations,
shapes,matrix,
positioning,
shapes.geometric,
arrows,through, graphs, graphs.standard
}
\usepackage{amsrefs}

\usepackage{blkarray}

\numberwithin{equation}{section}
\usetikzlibrary{positioning}
\usetikzlibrary{arrows}

\newtheorem{theorem}{Theorem}[section]
\newtheorem{lemma}[theorem]{Lemma}
\newtheorem{proposition}[theorem]{Proposition}
\newtheorem{corollary}[theorem]{Corollary}

\newtheorem{question}[theorem]{Question}

\theoremstyle{definition}
\newtheorem{definition}[theorem]{Definition} 
\newtheorem{remark}[theorem]{Remark}
\newtheorem{example}[theorem]{Example}

\newcommand{\K}{\mathbb{K}}
\newcommand{\C}{\mathbb{C}}
\newcommand{\N}{\mathbb{N}}

\newcommand{\R}{\mathbb{R}}

\newcommand{\m}{\mathfrak{m}}
\newcommand{\supp}{\text{supp}}
\newcommand{\tuple}[1]{\langle #1\rangle}
\newcommand{\st}{\colon}

\DeclareMathOperator{\sdefect}{sdefect}
\DeclareMathOperator{\rk}{\text{rank}~}
\DeclareMathOperator{\Ind}{\text{Ind}}

\newcommand{\qand}{\quad \mbox{and} \quad}

\newcommand{\F}{\mathcal{F}}
\newcommand{\Nn}{\mathcal{N}}

\begin{document}
 
\title{Lefschetz properties of squarefree monomial ideals via Rees algebras}

\author{Thiago Holleben}
\address[Thiago Holleben]
{Department of Mathematics \& Statistics,
Dalhousie University,
6316 Coburg Rd.,
PO BOX 15000,
Halifax, NS,
Canada B3H 4R2}
\email{hollebenthiago@dal.ca}

\keywords{Simplicial complexes, Lefschetz properties, Rees algebra, Symbolic powers}
\subjclass[2020]{13E10, 13A30, 13F55, 05E45, 05E40}

 
\begin{abstract}
The theory of Rees algebras of monomial ideals has been extensively studied, and as a consequence, many (sometimes partial) equivalences between algebraic properties of monomial ideals, and combinatorial properties of simplicial complexes and hypergraphs are known. In this paper we show how this theory can be used to find interesting examples in the theory of Lefschetz properties. We explore the consequences of known results from Lefschetz properties to the Rees algebras of squarefree monomial ideals, for example in the calculation of analytic spread. In particular, we show a connection between symbolic powers and $f$-vectors of simplicial complexes. This perspective leads us to a generalization of Postnikov's "mixed Eulerian numbers". We prove the positivity of such numbers in our setting.
\end{abstract}

\maketitle


\section{Introduction}
Rees algebras, a fundamental tool for the algebraic study of singularities, have played a major role in connecting Commutative Algebra to Combinatorics.
When $I$ is a monomial ideal, several combinatorial techniques have been applied to (at least partially) characterize properties of such ideals that are related to their Rees algebras. A common theme in this theory is reducing the problem of checking whether an ideal satisfies a given property to a problem involving linear algebra.   

A place where linear algebra appears prominently is the definition of the weak and strong Lefschetz properties (abbreviated as WLP and SLP).  
The study of Lefschetz properties has several connections to areas such as Geometry and Combinatorics.

In this paper we show how the rich and extensive theory of Rees algebras of (squarefree) monomial ideals manifests itself in the theory of Lefschetz properties.
For a simplicial complex $\Delta$, its Stanley-Reisner ideal $\Nn(\Delta)$, define the artinian reduction
$$ 
    A(\Delta) = \frac{\K[x_1, \dots, x_n]}{\Nn(\Delta) + (x_1^2, \dots, x_n^2)}.
$$
We show the following.

\begin{theorem}\label{t:resultsintro}
    Let $\Delta$ be a pure simplicial complex of dimension $d$, $I$ the facet ideal of $\Delta$ and $A(\Delta)$ as above. Then

    \begin{enumerate}
        \item (\textbf{Linear type and SLP}) If $d = 2$ and $I$ is of linear type, then $A(\Delta)$ has the strong Lefschetz property (SLP) in characteristic zero. If moreover every connected component of $\Delta$ is a $2$-dimensional simplicial tree, then $A(\Delta)$ has the SLP in every odd characteristic.
        \item (\textbf{Symbolic powers and SLP}) If $I$ has at least as many generators as variables (or equivalently, $\Delta$ has at least as many facets as vertices) and $I^{(m)} = I^m$ for every $m \geq 1$, then $A(\Delta)$ fails the SLP in characteristic zero.
        \item (\textbf{Level monomial algebras failing SLP}) If $\Delta$ is the cover complex of a whiskered bipartite graph on $n \geq 5$ vertices, then $A(\Delta)$ fails the SLP in characteristic zero.
    \end{enumerate}
\end{theorem}

We note that item $(3)$ in \cref{t:resultsintro} is related to a question of Migliore and Nagel~\cite[Question 4.1]{tour}, which asks which (if any) level monomial algebras fail the WLP or SLP, as it offers a way to construct a large family of level monomial algebras failing the SLP.

Going in the other direction, we study implications of known results from the theory of Lefschetz properties to the theory of (Symbolic) Rees algebras of squarefree monomial ideals. The symbolic defect, defined by Galetto and coauthors~\cite{MR3906569} is a tool to measure the difference between ordinary and symbolic powers of an ideal. 
In~\cite{lefschetzmm} we have shown how the pure skeleta of $\Delta$ affects the SLP of $A(\Delta)$.
In this paper, we develop the notion of symbolic defect polynomials to replicate the role of the symbolic defect, to measure the symbolic defect of a collection of ideals comming from the pure skeleta.
This perspective leads us to a different approach in the study of symbolic powers of squarefree monomial ideals and we show the following connection to the theory of $f$-vectors:
  
\begin{theorem}
    Let $\Delta$ be a $d$-dimensional flag simplicial complex and 
    $$
        \mu_2(\Delta, t) = c_2 t^3 + c_3 t^4 + \dots + c_{d} t^{d + 1}
    $$
    its second symbolic defect polynomial. Then $c_2 = f_2$ and $c_i \geq f_i$ for every $2 \leq i \leq d$, where $f_i$ denotes the number of $i$-faces of $\Delta$. 
    In particular, the coefficients of $\mu_2(\Delta, t)$ form a sequence with no internal zeros.
\end{theorem}

By applying a result of Hausel~\cite{hausel} on the SLP of level monomial algebras, we are able to show the following:

\begin{theorem}\label{t:analyticspreadintro}
    Let $\Delta$ be a pure simplicial complex of dimension $d$ with $n$ vertices, $S = \K[x_1,\dots, x_n]$
     and $I_i$ the ideal generated by the $i$-faces of $\Delta$, and $\m = (x_1, \dots, x_n)$.
    \begin{enumerate}
        \item For $0 \leq i < d$ the analytic spread of $I_i$ is $n$.
        \item  For every ${\bf a} = (a_0, \dots, a_{d - 1}) \in \N^d$ such that $a_0 + \dots + a_{d - 1} = n - 1$, the mixed multiplicity
        $$
            e_{{\bf a}}(\m | I_1, \dots, I_{d -1})
        $$
        is positive.
    \end{enumerate}
\end{theorem}

Item $(2)$ of \cref{t:analyticspreadintro} has natural connections to Combinatorics, and can be seen as a generalization of two out of the $9$ properties shown by Postnikov in~\cite{MR2487491}, on the positivity of "mixed Eulerian numbers". In \cref{s:mvmm} we focus on these connections, and finally in \cref{s:questions} we gather questions that follow from our work.

\section{Preliminaries: Simplicial complexes and hypergraphs}

A \textbf{simplicial complex} $\Delta$ on vertex set $V$ is a collection of subsets of $V$, such that if $\tau \in \Delta$ and $\sigma \subset \tau$, then $\sigma \in \Delta$. Elements in $\Delta$ are called \textbf{faces} and maximal faces are called \textbf{facets}. If the facets of $\Delta$ are $F_1, \dots, F_s$, we write $\Delta = \langle F_1, \dots, F_s \rangle$. The dimension of a face $\tau \in \Delta$ is $\dim \tau = |\tau| - 1$, $0$-dimensional faces are called \textbf{vertices}, $1$-dimensional faces are called \textbf{edges} and $2$-dimensional faces are called \textbf{triangles}. The dimension of $\Delta$ is 
$$
    \dim \Delta = \max (\dim \tau : \tau \in \Delta).
$$
If every facet of $\Delta$ has the same dimension, we say $\Delta$ is a \textbf{pure} simplicial complex. The \textbf{f-vector} of a $d$-dimensional complex $\Delta$ is $(f_{-1}(\Delta), f_0(\Delta), f_1(\Delta), \dots, f_d(\Delta))$, where $f_i(\Delta)$ is the number of $i$-dimensional faces of $\Delta$. We set $f_{-1}(\Delta) = 1$. If there is no confusion, we denote $f_i(\Delta)$ simply by $f_i$.

A specific class of simplicial complexes that was introduced by Faridi in \cite{faridifacet} and will appear in later sections is the following.

\begin{definition}[\textbf{Leaves and simplicial forests}]
    Let $\Delta = \langle F_1, \dots, F_s\rangle$ be a simplicial complex. We say a facet $F$ is a \textbf{leaf} of $\Delta$ if either $F$ is the only facet of $\Delta$, or there exists a facet $G \in \Delta$, $G \neq F$ such that
    $$
        F \cap F' \subseteq F \cap G
    $$
    for every facet $F' \in \Delta$, $F' \neq F$. We say $\Delta$ is a \textbf{simplicial tree} (or a tree) if $\Delta$ is connected and every subcomplex $\Delta' = \langle F_{i_1}, \dots, F_{i_r} \rangle$ has a leaf. If $\Delta$ every connected component of $\Delta$ is a tree, we say $\Delta$ is a \textbf{simplicial forest} (or a forest).
\end{definition}

\begin{center}

\tikzset{every picture/.style={line width=0.75pt}} 

\begin{tikzpicture}[x=0.75pt,y=0.75pt,yscale=-1,xscale=1]

\draw  [fill={rgb, 255:red, 155; green, 155; blue, 155 }  ,fill opacity=1 ] (104.5,130) -- (139.5,162) -- (96,181) -- (66,161) -- (56.5,120) -- (104.5,130) -- (145.5,122) -- (139.5,162) -- cycle ;
\draw    (66,161) -- (104.5,130) ;
\draw    (96,181) -- (104.5,130) ;
\draw  [fill={rgb, 255:red, 155; green, 155; blue, 155 }  ,fill opacity=1 ] (308.5,139) -- (308.5,180) -- (258.5,180) -- (258.5,139) -- cycle ;
\draw    (258.5,139) -- (308.5,180) ;
\draw    (308.5,139) -- (258.5,180) ;
\draw  [fill={rgb, 255:red, 155; green, 155; blue, 155 }  ,fill opacity=1 ] (423.88,148) -- (448.75,179) -- (399,179) -- cycle ;
\draw  [fill={rgb, 255:red, 155; green, 155; blue, 155 }  ,fill opacity=1 ] (448.75,117) -- (473.63,148) -- (423.88,148) -- cycle ;
\draw  [fill={rgb, 255:red, 155; green, 155; blue, 155 }  ,fill opacity=1 ] (473.63,148) -- (498.5,179) -- (448.75,179) -- cycle ;

\draw (53,193) node [anchor=north west][inner sep=0.75pt]   [align=left] {Simplicial tree};
\draw (220,193) node [anchor=north west][inner sep=0.75pt]   [align=left] {Not a simplicial tree};
\draw (385,193) node [anchor=north west][inner sep=0.75pt]   [align=left] {Not a simplicial tree};

\end{tikzpicture}
\end{center}

A (simple) \textbf{$m$-hypergraph} $H = (V, E)$ is a tuple where $V$ is the set of vertices, and $E$ is a collection of subsets of size $m$ of $V$ such that if $e, e' \in E$, then $e \not \subset e'$ and $e' \not \subset e$. Given a pure $d$-dimensional simplicial complex $\Delta = \langle F_1, \dots, F_s \rangle$ on vertex set $V$, we can associate to it the following $d + 1$-hypergraph:
$$
    H = (V, E), \text{ where } E = \{F_1, \dots, F_s\}.
$$
On the other hand, given an $m$-hypergraph $H = (V, \{e_1, \dots, e_r\})$, we can associate to it the simplicial complex $\Delta = \langle e_1, \dots, e_r \rangle$. In particular, we can go back and forth between $m$-hypergraphs and pure simplicial complexes.

\begin{definition}
    Let $H = (V, E)$ be an $m$-hypergraph. The \textbf{incidence matrix} of $H$ is the $|E| \times |V|$ matrix $M(H)$ such that 
    $$
        M(H)_{ij} = \begin{cases}
            1 \ \ \ \text{ if the $i$-th edge of $H$ contains the $j$-th vertex of $H$} \\
            0 \ \ \ \text{ otherwise}
        \end{cases}
    $$
    $H$ is said to be \textbf{unimodular} if every minor of $M(H)$ is $-1, 0$ or $1$.
\end{definition} 

We will need the following particular case of a theorem by Berge:

\begin{definition}
    Let $H = (V, E)$ be an $m$-hypergraph and $r \geq 2$ an integer. A \textbf{cycle} of length $r$ is a sequence $(x_1, E_1, x_2, E_2, x_3, \dots, x_r, E_r, x_1)$ with:
    \begin{enumerate}
        \item $E_1, E_2, \dots, E_r$ are distinct edges of $H$
        \item $x_1, x_2, \dots, x_r$ are distinct vertices of $H$
        \item $x_i, x_{i + 1} \in E_i$ for $i = 1, \dots, r - 1$
        \item $x_r x_1 \in E_r$
    \end{enumerate}
    An \textbf{odd cycle} is a cycle of odd length.
\end{definition}

\begin{theorem}[\cite{MR1013569}]\label{t:unimodular}
    Every $m$-hypergraph without odd cycles is unimodular. In particular, bipartite graphs are unimodular.
\end{theorem}

The property of unimodular graphs that will become useful in later sections is the fact that when $H$ is a unimodular hypergraph, the greatest common divisor of every maximal minor of the incidence matrix of $H$ is either $0$ or $1$. Before moving on, we prove that simplicial forests also have this property.

\begin{lemma}\label{l:lessfacetsforests}
    If $\Delta = \tuple{F_1, \dots, F_s}$ is a $d$-dimensional simplicial forest, then $f_0 \geq f_d$.
\end{lemma}

\begin{proof}
    We will first show that the lemma holds when $\Delta$ is pure. Every leaf $F_i$ of a forest has a free vertex (see for example~\cite{faridiforest}), that is, a vertex $v$ such that the only facet of $\Delta$ that contains $v$ is $F_i$.
    By removing the facet $F_i$, we have a new simplicial forest $\Delta \setminus F_i = \tuple{F_1, \dots, F_{i - 1}, F_{i + 1}, \dots, F_s}$ where $f_0(\Delta \setminus F_i) \leq f_0(\Delta) - 1$ and $f_d(\Delta \setminus F_i) = f_d(\Delta) - 1$. Repeating this process until the complex is a single simplex $\Delta'$, we see that 
    $$
        d + 1 = f_0(\Delta') \geq f_d(\Delta') = 1
    $$
    and since at each step the number of vertices decreased by at least one, while the number of facets decreased by exactly one, the result follows.

    If $\Delta$ is not pure, then consider the subcomplex $\Delta' = \tuple{G_1, \dots, G_r}$ of $\Delta$, where $G_1, \dots, G_r$ are exactly the $d$-faces of $\Delta$. By the arguments above, and since $\Delta'$ is a pure simplicial forest, we conclude $f_0(\Delta') \geq f_d(\Delta')$. The result then follows by noticing that $f_0(\Delta) \geq f_0(\Delta') \geq f_d(\Delta') = f_d(\Delta)$.
\end{proof}

\begin{proposition}[\textbf{Greatest common divisor of minors of incidence matrix of simplicial forests}]\label{p:forestgcdminors}
    Let $\Delta$ be a simplicial forest. The greatest common divisor of the maximal minors of the incidence matrix of $\Delta$ is $1$.
\end{proposition}
\begin{proof}
    Let $\Delta = \tuple{F_1, \dots, F_s}$. Since every subcomplex of $\Delta$ of the form $\tuple{F_{i_1},\dots, F_{i_r}}$ is a simplicial forest, we may assume the facets of $\Delta$ are ordered such that $F_i$ is a leaf of $\tuple{F_i, F_{i + 1}, \dots, F_s}$. In particular, for every $i$, $F_i$ contains a vertex $v_i$ such that if $v_i \in F_j$, then $j \leq i$. 

    By \cref{l:lessfacetsforests}, a maximal minor of the incidence matrix of $\Delta$ corresponds to picking a set of $s$ vertices of $\Delta$. Consider the maximal minor where the selected vertices are exactly $v_1, \dots, v_s$. Then by the definition of $v_1, \dots, v_s$, after reordering the columns, this submatrix is an upper triangular matrix where every diagonal entry is $1$, in particular the incidence matrix of $\Delta$ has at least one maximal minor equal to $1$, so the gcd of maximal minors is also $1$.
\end{proof}

Given a simplicial complex $\Delta$ on vertex set $[n] = \{1,\dots, n\}$, we can associate to it two different ideals.

\begin{definition}[\textbf{Stanley-Reisner and Facet ideals}] 
    Let $\Delta = \langle F_1, \dots, F_s\rangle$ be a simplicial complex on vertex set $[n]$, and $S = \K[x_1, \dots, x_n]$. For $\tau \subset [n]$, define $x_\tau = \prod_{j \in \tau} x_j$.

    \begin{enumerate}
        \item The \textbf{facet ideal} of $\Delta$ is the ideal $\F(\Delta) = (x_{F_1}, \dots, x_{F_s}) \subset S$. Given a squarefree monomial ideal $I \subset S$, the simplicial complex $\Delta$ such that $I = \F(\Delta)$ is called the \textbf{facet complex} of $I$.
        \item The \textbf{Stanley-Reisner ideal} of $\Delta$ is the ideal $\Nn(\Delta) = (x_\tau : \tau \not \in \Delta) \subset S$. Given a squarefree monomial ideal $I \subset S$, the simplicial complex $\Delta$ such that $I = \Nn(\Delta)$ is called the \textbf{Stanley-Reisner complex} of $I$.
    \end{enumerate}
\end{definition}

\begin{example}\label{e:stanleyfacet}
    Let $I = (ab, bc, ad, be, cf) \subset \K[a,b,c,d,e,f]$. Then    
\begin{center}
    
\tikzset{every picture/.style={line width=0.75pt}} 

\begin{tikzpicture}[x=0.75pt,y=0.75pt,yscale=-1,xscale=1]

\draw    (266.5,88) -- (314.5,129) ;
\draw [shift={(314.5,129)}, rotate = 40.5] [color={rgb, 255:red, 0; green, 0; blue, 0 }  ][fill={rgb, 255:red, 0; green, 0; blue, 0 }  ][line width=0.75]      (0, 0) circle [x radius= 3.35, y radius= 3.35]   ;
\draw [shift={(266.5,88)}, rotate = 40.5] [color={rgb, 255:red, 0; green, 0; blue, 0 }  ][fill={rgb, 255:red, 0; green, 0; blue, 0 }  ][line width=0.75]      (0, 0) circle [x radius= 3.35, y radius= 3.35]   ;
\draw    (314.5,129) -- (357.5,87) ;
\draw [shift={(357.5,87)}, rotate = 315.67] [color={rgb, 255:red, 0; green, 0; blue, 0 }  ][fill={rgb, 255:red, 0; green, 0; blue, 0 }  ][line width=0.75]      (0, 0) circle [x radius= 3.35, y radius= 3.35]   ;
\draw [shift={(314.5,129)}, rotate = 315.67] [color={rgb, 255:red, 0; green, 0; blue, 0 }  ][fill={rgb, 255:red, 0; green, 0; blue, 0 }  ][line width=0.75]      (0, 0) circle [x radius= 3.35, y radius= 3.35]   ;
\draw    (266.5,88) -- (267.5,152) ;
\draw [shift={(267.5,152)}, rotate = 89.1] [color={rgb, 255:red, 0; green, 0; blue, 0 }  ][fill={rgb, 255:red, 0; green, 0; blue, 0 }  ][line width=0.75]      (0, 0) circle [x radius= 3.35, y radius= 3.35]   ;
\draw [shift={(266.5,88)}, rotate = 89.1] [color={rgb, 255:red, 0; green, 0; blue, 0 }  ][fill={rgb, 255:red, 0; green, 0; blue, 0 }  ][line width=0.75]      (0, 0) circle [x radius= 3.35, y radius= 3.35]   ;
\draw    (357.5,154) -- (357.5,87) ;
\draw [shift={(357.5,87)}, rotate = 270] [color={rgb, 255:red, 0; green, 0; blue, 0 }  ][fill={rgb, 255:red, 0; green, 0; blue, 0 }  ][line width=0.75]      (0, 0) circle [x radius= 3.35, y radius= 3.35]   ;
\draw [shift={(357.5,154)}, rotate = 270] [color={rgb, 255:red, 0; green, 0; blue, 0 }  ][fill={rgb, 255:red, 0; green, 0; blue, 0 }  ][line width=0.75]      (0, 0) circle [x radius= 3.35, y radius= 3.35]   ;
\draw    (314.5,129) -- (315.5,179) ;
\draw [shift={(315.5,179)}, rotate = 88.85] [color={rgb, 255:red, 0; green, 0; blue, 0 }  ][fill={rgb, 255:red, 0; green, 0; blue, 0 }  ][line width=0.75]      (0, 0) circle [x radius= 3.35, y radius= 3.35]   ;
\draw [shift={(314.5,129)}, rotate = 88.85] [color={rgb, 255:red, 0; green, 0; blue, 0 }  ][fill={rgb, 255:red, 0; green, 0; blue, 0 }  ][line width=0.75]      (0, 0) circle [x radius= 3.35, y radius= 3.35]   ;
\draw  [fill={rgb, 255:red, 155; green, 155; blue, 155 }  ,fill opacity=1 ] (86,119) -- (171,119) -- (171,204) -- (86,204) -- cycle ;
\draw  [fill={rgb, 255:red, 155; green, 155; blue, 155 }  ,fill opacity=1 ] (128.75,63) -- (171.5,119) -- (86,119) -- cycle ;
\draw    (86,119) -- (171,204) ;
\draw    (171,119) -- (86,204) ;

\draw (123,42) node [anchor=north west][inner sep=0.75pt]   [align=left] {$\displaystyle b$};
\draw (131,155) node [anchor=north west][inner sep=0.75pt]   [align=left] {$\displaystyle e$};
\draw (173,99) node [anchor=north west][inner sep=0.75pt]   [align=left] {$\displaystyle d$};
\draw (73,102) node [anchor=north west][inner sep=0.75pt]   [align=left] {$\displaystyle f$};
\draw (73,202) node [anchor=north west][inner sep=0.75pt]   [align=left] {$\displaystyle a$};
\draw (173,200) node [anchor=north west][inner sep=0.75pt]   [align=left] {$\displaystyle c$};
\draw (253,65) node [anchor=north west][inner sep=0.75pt]   [align=left] {$\displaystyle a$};
\draw (310,106) node [anchor=north west][inner sep=0.75pt]   [align=left] {$\displaystyle b$};
\draw (365,69) node [anchor=north west][inner sep=0.75pt]   [align=left] {$\displaystyle c$};
\draw (261,157) node [anchor=north west][inner sep=0.75pt]   [align=left] {$\displaystyle d$};
\draw (311,185) node [anchor=north west][inner sep=0.75pt]   [align=left] {$\displaystyle e$};
\draw (351,160) node [anchor=north west][inner sep=0.75pt]   [align=left] {$\displaystyle f$};
\draw (51,230) node [anchor=north west][inner sep=0.75pt]   [align=left] {\begin{minipage}[lt]{98.54pt}\setlength\topsep{0pt}
\begin{center}
The Stanley-Reisner \\complex of $\displaystyle I$
\end{center}

\end{minipage}};
\draw (250,230) node [anchor=north west][inner sep=0.75pt]   [align=left] {\begin{minipage}[lt]{104.08pt}\setlength\topsep{0pt}
\begin{center}
The facet complex of $\displaystyle I$
\end{center}

\end{minipage}};

\end{tikzpicture}
\end{center}
\end{example}
 
Let $I$ be an ideal of $S = \K[x_1, \dots, x_n]$. The ring $S[It] = \oplus_{i \in \N} I^i t^i \subset S[t]$ is called the \textbf{Rees ring} of $I$. Let $\m$ be the maximal homogeneous ideal of $S$. The Krull dimension of $S[It]/\m S[It]$, denoted by $\ell(I)$, is called the \textbf{analytic spread} of $I$. 

Given a monomial ideal $I = (x_1^{a_{1,1}}\dots x_n^{a_{1,n}}, \dots, x_1^{a_{s, 1}} \dots x_n^{a_{s, n}})$ generated by monomials of degree $t$, the \textbf{log-matrix} of $I$ is 
$$
    \log(I) = \begin{pmatrix}
        a_{1,1} & \dots &  a_{1,n} \\
        \vdots  & \ddots & \vdots \\
        a_{s, 1} & \dots & a_{s, n}
    \end{pmatrix}
$$
that is, rows of $\log(I)$ are the exponents of the generators of $I$. The following standard result relates the rank of $\log(I)$ and $\ell(I)$.

\begin{theorem}[\cite{rankanalyticspread}]\label{rankanalyticspread}
    Let $I \subset S$ be a monomial ideal generated by monomials of degree $d$. Then 
    $$
        \ell(I) = \rk \log(I).
    $$
\end{theorem}

\Cref{rankanalyticspread} gives us a way of studying the rank of matrices using Rees algebras and vice-versa. The following family of ideals will play an important role in future sections.

\begin{definition}
    An ideal $I \subset S$ is said to be of \textbf{linear type} if the Rees algebra $S[It]$ of $I$ is isomorphic to the symmetric algebra $\text{Sym}(I)$ of $I$, or equivalently, $S[It]$ is defined as an $S[w_0,\dots,w_s]$-algebra by linear polynomials on $w_0,\dots, w_s$. 
\end{definition}

\begin{lemma}\label{l:lineartype}
    If an ideal $I \subset S$ is of linear type, and is minimally generated by $s$ elements, then $\ell(I) = s$.
\end{lemma}

\begin{example}\label{e:notlineartype}
    Let $I = (aec, cde, def, aef, bdf) \subset S = \K[a,b,c,d,e,f]$ be the facet ideal of the complex in \cref{e:stanleyfacet}. The Rees algebra of $I$ is isomorphic to 
    $$
        S[It] \cong \frac{S[w_0,\dots, w_4]}{(aw_3 - dw_4, cw_1 - fw_4, cw_0 - fw_3, bw_0 - e w_2, a w_0 - dw_1, w_1 w_3 - w_0 w_4)}.
    $$
    In particular, $w_1 w_3 - w_0 w_4$ is not linear on $w_0,\dots, w_4$ and therefore $I$ is not of linear type. As we will see in~\cref{e:wslp}, $\ell(I) < 5$.
\end{example}

\section{Lefschetz properties}

Let $\K$ be an infinite field, $S = \K[x_1, \dots, x_n]$ and $I$ a homogeneous ideal of $S$ such that $A = S/I$ is an artinian graded algebra. The set $\text{soc}(A) = \{f : f A_1 = 0\}$ is called the \textbf{socle} of $A$. If $\text{soc}(A) = A_d$, where $d = \max\{i : A_i \neq 0\}$, then $A$ is said to be \textbf{level} and the integer $d$ is called the \textbf{socle degree} of $A$.

\begin{definition}
    Let $L \in S_1$ be a general linear form. The algebra $A$ is said to satisfy the \textbf{weak Lefschetz property} (WLP) if the multiplication maps $\times L: A_i \to A_{i + 1}$ have maximal rank for every $i = 0, \dots, d - 1$. 
    
    We say $A$ satisfies the \textbf{strong Lefschetz Property} (SLP) if the maps $\times L^j: A_i \to A_{i + j}$ have full rank for every $i, j$.
\end{definition}

When the algebra $A$ is defined by a monomial ideal $I \subset S$, we have the following very useful lemma.

\begin{lemma}[\cite{wlpmon}, Proposition 2.3]
    Let $I \subset S = \K[x_1, \dots, x_n]$ be a monomial ideal such that $A = S/I$ is an artinian algebra and $\K$ an infinite field. Then $A$ has the WLP (resp. SLP) if and only if the multiplication maps by $L = x_1 + \dots + x_n$ (resp. powers of $L$) have full rank.
\end{lemma}

Given a simplicial complex $\Delta$ on vertex set $[n]$, we set 
$$
    A(\Delta) := \frac{S}{(\Nn(\Delta), x_1^2, \dots, x_n^2)}.
$$
The study of Lefschetz properties of algebras of the form $A(\Delta)$ has recently attracted attention from researchers (see for example~\cite{lefschetzmm,cooper2023weak,kling2023strong,nguyen2023weak}). One motivation for this is the fact that the algebra $A(\Delta)$ contains all the combinatorial information of $\Delta$.

\begin{proposition}[\cite{daonair, mns}]
    Let $\Delta$ be a simplicial complex. The monomials of $A(\Delta)$ of degree $i$ are in bijection with the $i-1$-dimensional faces of $\Delta$. The Hilbert series of $A(\Delta)$ is given by:
    $$
        HS_{A(\Delta)}(t) = \sum_{i \geq 0} f_{i - 1}t^i.
    $$
    In particular, the $h$-vector of $A(\Delta)$ is the $f$-vector of $\Delta$.
\end{proposition}

\begin{proposition}[\cite{boijthesis}]\label{p:purelevel}
    Let $\Delta$ be a simplicial complex. The algebra $A(\Delta)$ is level if and only if $\Delta$ is pure.
\end{proposition}

\begin{example}\label{e:wslp}
    Let $\Delta$ be the Stanley-Reisner complex of the ideal $I$ from example \cref{e:stanleyfacet}. The results in~\cite{cooper2023weak} imply $A(\Delta)$ has the WLP when the base field has any odd characteristic. The matrix that represents the map $\times L^2: A(\Delta)_1 \to A(\Delta)_3$, where $L = a + b + c + d + e + f$ is
    $$
    M = 
    \begin{blockarray}{ccccccc}
        & a & b & c & d & e & f \\
      \begin{block}{c(cccccc)}
        aef & 2 & 0 & 0 & 0 & 2 & 2 \\
        ace & 2 & 0 & 2 & 0 & 2 & 0 \\
        ecd & 0 & 0 & 2 & 2 & 2 & 0 \\
        edf & 0 & 0 & 0 & 2 & 2 & 2 \\
        bdf & 0 & 2 & 0 & 2 & 0 & 2 \\
      \end{block}
    \end{blockarray}.
    $$
    In particular, $M$ is not surjective since its transpose is not injective
    $$
        M^T \begin{pmatrix}
            1 \\
            -1 \\
            1 \\
            -1 \\
            0
        \end{pmatrix} = 0
    $$
    and so $A(\Delta)$ fails the SLP. Note that the element in the kernel above comes from the cycle of even length of $\Delta$, namely $e, \{a,e,f\}, a, \{a,c,e\}, c, \{e,c,d\}, d, \{e,d,f\}, e$.
\end{example}
  
Next we define incidence ideals of a simplicial complex $\Delta$, which were first introduced in \cite{lefschetzmm} to get information on the failure of the WLP  of $A(\Delta)$ in positive characteristics.

\begin{definition}
    The \textbf{incidence ring} of a simplicial complex $\Delta$ is the ring:
    $$
        R_\Delta = \C[t_\sigma : \sigma \in \Delta].
    $$
    Moreover, we set $R_{\Delta, i} = R_\Delta/(t_\sigma : |\sigma| \neq i)$. In other words, $R_\Delta$ is a polynomial ring where the variables are the faces of $\Delta$, and $R_{\Delta, i}$ is the polynomial ring where the variables are the $i-1$-dimensional faces of $\Delta$.

    Let $\sigma_1, \dots, \sigma_{f_j}$ be the $j$-faces of $\Delta$. The \textbf{$(i, j)$-th incidence ideal of $\Delta$} is the ideal
    $$
        I_\Delta(i, j) = \Bigg{(}\prod_{\tau \subset \sigma_1, |\tau| = i + 1} t_\tau, \dots, \prod_{\tau \subset \sigma_{f_{j}}, |\tau| = i + 1} t_\tau\Bigg{)} \subset R_{\Delta, i + 1}.
    $$
    That is, the generators of $I_\Delta(i, j)$ are the products of sets of variables in $R_{\Delta, i + 1}$ corresponding to the $i$-faces of a $j$-face of $\Delta$.
\end{definition}

\begin{definition}
    Let $\Delta$ be a simplicial complex, the $i$-th \textbf{skeleton} of $\Delta$, denoted by $\Delta^{(i)}$, is the simplicial complex where the faces of $\Delta^{(i)}$ are the faces of $\Delta$ that have dimension at most $i$.

    We define the \textbf{$(i, j)$-incidence hypergraph} of $\Delta$, denoted by $H_{\Delta}(i, j)$, as the hypergraph with vertex set the $i$-dimensional faces of $\Delta$ and a set $E$ of $i$-faces forms an edge if and only if $E$ is the set of $i$-faces of a $j$-face of $\Delta$. Note that $H_\Delta(i, j)$ is also a simplicial complex (the edges of the hypergraph being the facets), which we will call the \textbf{$(i,j)$-incidence complex} of $\Delta$.
\end{definition}

\begin{example}\label{e:incidencecomplexes}
    Let $\Delta$ be the Stanley-Reisner complex of $I$ from \cref{e:stanleyfacet}. Then the complexes below are the incidence complexes of $\Delta$.
    \begin{center}

\tikzset{every picture/.style={line width=0.75pt}} 

\begin{tikzpicture}[x=0.75pt,y=0.75pt,yscale=-1,xscale=1]

\draw  [fill={rgb, 255:red, 155; green, 155; blue, 155 }  ,fill opacity=1 ] (472,104) -- (557,104) -- (557,189) -- (472,189) -- cycle ;
\draw  [fill={rgb, 255:red, 155; green, 155; blue, 155 }  ,fill opacity=1 ] (514.75,48) -- (557.5,104) -- (472,104) -- cycle ;
\draw    (472,104) -- (557,189) ;
\draw    (557,104) -- (472,189) ;
\draw    (82,104) -- (167,189) ;
\draw    (167,104) -- (82,189) ;
\draw    (82,104) -- (82,189) ;
\draw    (167,189) -- (82,189) ;
\draw    (124.5,146.5) -- (82,189) ;
\draw [shift={(82,189)}, rotate = 135] [color={rgb, 255:red, 0; green, 0; blue, 0 }  ][fill={rgb, 255:red, 0; green, 0; blue, 0 }  ][line width=0.75]      (0, 0) circle [x radius= 3.35, y radius= 3.35]   ;
\draw [shift={(124.5,146.5)}, rotate = 135] [color={rgb, 255:red, 0; green, 0; blue, 0 }  ][fill={rgb, 255:red, 0; green, 0; blue, 0 }  ][line width=0.75]      (0, 0) circle [x radius= 3.35, y radius= 3.35]   ;
\draw    (82,104) -- (124.5,146.5) ;
\draw [shift={(124.5,146.5)}, rotate = 45] [color={rgb, 255:red, 0; green, 0; blue, 0 }  ][fill={rgb, 255:red, 0; green, 0; blue, 0 }  ][line width=0.75]      (0, 0) circle [x radius= 3.35, y radius= 3.35]   ;
\draw [shift={(82,104)}, rotate = 45] [color={rgb, 255:red, 0; green, 0; blue, 0 }  ][fill={rgb, 255:red, 0; green, 0; blue, 0 }  ][line width=0.75]      (0, 0) circle [x radius= 3.35, y radius= 3.35]   ;
\draw    (167,104) -- (124.5,146.5) ;
\draw    (167,104) -- (167,189) ;
\draw [shift={(167,189)}, rotate = 90] [color={rgb, 255:red, 0; green, 0; blue, 0 }  ][fill={rgb, 255:red, 0; green, 0; blue, 0 }  ][line width=0.75]      (0, 0) circle [x radius= 3.35, y radius= 3.35]   ;
\draw [shift={(167,104)}, rotate = 90] [color={rgb, 255:red, 0; green, 0; blue, 0 }  ][fill={rgb, 255:red, 0; green, 0; blue, 0 }  ][line width=0.75]      (0, 0) circle [x radius= 3.35, y radius= 3.35]   ;
\draw    (124.75,48) -- (82,104) ;
\draw    (124.75,48) -- (167,104) ;
\draw [shift={(167,104)}, rotate = 52.97] [color={rgb, 255:red, 0; green, 0; blue, 0 }  ][fill={rgb, 255:red, 0; green, 0; blue, 0 }  ][line width=0.75]      (0, 0) circle [x radius= 3.35, y radius= 3.35]   ;
\draw [shift={(124.75,48)}, rotate = 52.97] [color={rgb, 255:red, 0; green, 0; blue, 0 }  ][fill={rgb, 255:red, 0; green, 0; blue, 0 }  ][line width=0.75]      (0, 0) circle [x radius= 3.35, y radius= 3.35]   ;
\draw    (82,104) -- (167,104) ;
\draw  [fill={rgb, 255:red, 155; green, 155; blue, 155 }  ,fill opacity=1 ] (350,73) -- (385.5,113) -- (314.5,113) -- cycle ;
\draw  [fill={rgb, 255:red, 155; green, 155; blue, 155 }  ,fill opacity=1 ] (425.35,146.43) -- (385.2,179.5) -- (385.5,113) -- cycle ;
\draw  [fill={rgb, 255:red, 155; green, 155; blue, 155 }  ,fill opacity=1 ] (348.24,218.84) -- (312.71,178.21) -- (385.2,179.5) -- cycle ;
\draw  [fill={rgb, 255:red, 155; green, 155; blue, 155 }  ,fill opacity=1 ] (274.16,143.29) -- (315.53,111.77) -- (312.71,178.21) -- cycle ;
\draw  [fill={rgb, 255:red, 155; green, 155; blue, 155 }  ,fill opacity=1 ] (274.16,143.29) -- (234.54,176.99) -- (233.78,110.5) -- cycle ;

\draw (509,27) node [anchor=north west][inner sep=0.75pt]   [align=left] {$\displaystyle b$};
\draw (521,140) node [anchor=north west][inner sep=0.75pt]   [align=left] {$\displaystyle e$};
\draw (559,84) node [anchor=north west][inner sep=0.75pt]   [align=left] {$\displaystyle d$};
\draw (459,87) node [anchor=north west][inner sep=0.75pt]   [align=left] {$\displaystyle f$};
\draw (459,187) node [anchor=north west][inner sep=0.75pt]   [align=left] {$\displaystyle a$};
\draw (559,185) node [anchor=north west][inner sep=0.75pt]   [align=left] {$\displaystyle c$};
\draw (119,27) node [anchor=north west][inner sep=0.75pt]   [align=left] {$\displaystyle b$};
\draw (131,140) node [anchor=north west][inner sep=0.75pt]   [align=left] {$\displaystyle e$};
\draw (169,84) node [anchor=north west][inner sep=0.75pt]   [align=left] {$\displaystyle d$};
\draw (69,87) node [anchor=north west][inner sep=0.75pt]   [align=left] {$\displaystyle f$};
\draw (65,187) node [anchor=north west][inner sep=0.75pt]   [align=left] {$\displaystyle a$};
\draw (174,187) node [anchor=north west][inner sep=0.75pt]   [align=left] {$\displaystyle c$};
\draw (217,178) node [anchor=north west][inner sep=0.75pt]   [align=left] {$\displaystyle bf$};
\draw (219,91) node [anchor=north west][inner sep=0.75pt]   [align=left] {$\displaystyle bd$};
\draw (261.16,147.52) node [anchor=north west][inner sep=0.75pt]   [align=left] {$\displaystyle fd$};
\draw (295.16,179.52) node [anchor=north west][inner sep=0.75pt]   [align=left] {$\displaystyle ef$};
\draw (298.16,92.52) node [anchor=north west][inner sep=0.75pt]   [align=left] {$\displaystyle de$};
\draw (339.16,52.52) node [anchor=north west][inner sep=0.75pt]   [align=left] {$\displaystyle cd$};
\draw (387.16,93.52) node [anchor=north west][inner sep=0.75pt]   [align=left] {$\displaystyle ce$};
\draw (427.16,137.52) node [anchor=north west][inner sep=0.75pt]   [align=left] {$\displaystyle ac$};
\draw (384.2,176.5) node [anchor=north west][inner sep=0.75pt]   [align=left] {$\displaystyle ae$};
\draw (339.16,218.52) node [anchor=north west][inner sep=0.75pt]   [align=left] {$\displaystyle af$};
\draw (93,243) node [anchor=north west][inner sep=0.75pt]   [align=left] {$\displaystyle H_{\Delta }( 0,1)$};
\draw (306,243) node [anchor=north west][inner sep=0.75pt]   [align=left] {$\displaystyle H_{\Delta }( 1,2)$};
\draw (484,243) node [anchor=north west][inner sep=0.75pt]   [align=left] {$\displaystyle H_{\Delta }( 0,2)$};

\end{tikzpicture}
    \end{center}
    Moreover, the incidence ideals are
    \begin{enumerate}
        \item $I_\Delta(0,1) = (t_a t_e, t_a t_f, t_a t_c, t_e t_c, t_e t_d, t_c t_d, t_e t_f, t_f t_d, t_b t_f, t_b t_d) \subset R_{\Delta, 1}$
        \item $I_\Delta(1, 2) = (t_{bf}t_{bd}t_{fd}, t_{fd} t_{de} t_{ef}, t_{ef} t_{af} t_{ae}, t_{ae} t_{ac} t_{ce}, t_{ce} t_{cd} t_{de}) \subset R_{\Delta, 2}$
        \item $I_\Delta(0,2) = (t_a t_e t_c, t_a t_e t_f, t_e t_c t_d, t_e t_d t_f, t_b t_d t_f) \subset R_{\Delta, 1}$.
    \end{enumerate}
\end{example}

\begin{remark}
    In \cite{lefschetzmm} incidence ideals were only defined for maps going from degree $i$ to $i + 1$, here we define them for maps going from degree $i$ to $i + j$. In the notation from \cite{lefschetzmm} we have $I_\Delta(i) = I_\Delta(i - 1, i)$.
\end{remark}

\begin{proposition}\cite{lefschetzmm}\label{lefschetzmm}
    Let $\Delta$ be a simplicial complex on vertex set $[n]$, $A(\Delta)$ the respective algebra and $L = x_1 + \dots + x_n$. Then up to a constant, the multiplication map $\times L^j: A(\Delta)_i \to A(\Delta)_{i + j}$ is represented by the incidence matrix of $H_\Delta(i - 1, i + j - 1)$, where the entries are taken in the base field of $A(\Delta)$. In particular, the map $\times L^j: A(\Delta)_i \to A(\Delta)_{i + j}$ has full rank in characteristic zero if and only if 
    $$
        \ell(\F(H_\Delta(i - 1, i + j - 1))) = \min (\dim A(\Delta)_i, \dim A(\Delta)_{i + j}).
    $$
\end{proposition}

\begin{proof}
        Note that for any nonzero monomial $m = x_{r_1} \dots x_{r_{i}} \in A(\Delta)_{i}$ we have
        $$
            L^j m = j! m\sum_{\{s_1, \dots, s_j\} \in \Delta} x_{s_1} \dots x_{s_j} = j! \sum_{\substack{\{s_1, \dots, s_{i + j}\} \in \Delta\\
            \{r_1, \dots, r_{i}\} \subset \{s_1, \dots, s_{i + j}\}}} x_{s_1} \dots x_{s_{i + j}}.
        $$
        In other words, after fixing an order for the monomial basis of $A(\Delta)_{i}$ and $A(\Delta)_{i + j}$, the matrix that represents the linear transformation $\times L^j: A(\Delta)_{i} \to A(\Delta)_{i + j}$ is given by
        $$
        [\times L^j]_{kl} = \begin{cases}
            j! \quad \text{if the $k$-th $i + j - 1$-face of $\Delta$ contains the $l$-th $i-1$-face of $\Delta$} \\
            0 \quad \text{otherwise},
        \end{cases}
        $$
        so $[\times L^j]$ is $j!$ times the incidence matrix of $H_\Delta(i - 1, i + j - 1)$. The last part of the statement follows since the rows of the matrix defined above all add up to $\binom{i + j}{i}$, which is the number of $i-1$-faces of a $i + j - 1$-face. So $H_\Delta(i - 1, i + j - 1)$ is a pure simplicial complex and its facet ideal is generated in one degree. By \cref{rankanalyticspread} we conclude 
        $$
            \ell(\F(H_\Delta(i - 1, i + j - 1))) = \rk [\times L^j] \leq \min(\dim A(\Delta)_i, \dim A(\Delta)_{i + j}),
        $$
        and so the result follows.
\end{proof}

The following criteria for the first multiplication map of $A(\Delta)$ to have full rank was first proved in \cite{daonair} in characteristic zero and then extended in \cite{lefschetzmm} to arbitrary monomial artinian algebras in odd characteristics.

\begin{theorem}[\cite{daonair,lefschetzmm}]\label{t:daonair}
    Let $\Delta$ be a simplicial complex on vertex set $[n]$ with at least $n$ edges and $L = x_1 + \dots + x_n$ a linear form in $S = \K[x_1, \dots, x_n]$. The multiplication map $\times L: A(\Delta)_1 \to A(\Delta)_2$ has full rank if and only if the characteristic of the base field is not $2$ and every connected component of $\Delta^{(1)}$ is not bipartite.
\end{theorem}
When dealing with results regarding Lefschetz properties, we will often assume our simplicial complex has at least as many vertices as faces of a specific dimension. This assumption means the result on Lefschetz properties implies injectivity. Some results, such as \cref{t:daonair}, can be modified to also include the surjective case, but in some cases, the behaviour of the injective and surjective cases are completely different (see for example~\cref{r:atleast}).

Lastly, we mention a result by Hausel \cite{hausel}, that shows the injectivity of multiplication maps in many degrees for monomial level artinian algebras. This result will be used in \cref{s:symbolic}.

\begin{theorem}[Theorem 6.3~\cite{hausel}]\label{hausel}
    Let $A = S/I$ be an artinian level $\K$-algebra of socle degree $d$ defined by a monomial ideal $I$, and $\K$ a field of characteristic zero. Let $L = x_1 + \dots + x_n \in S_1$. The following maps are injective
    $$
        \times L^{d - 2k}: A_k \to A_{d - k} \quad \text{for $k < d/2$.}
    $$
\end{theorem}

\begin{corollary}\label{c:maxanalyticspread}
    Let $\Delta$ be a pure simplicial complex of dimension $d$ on vertex set $[n]$. Then for every $1 \leq i < d$, the ideals $\F(\Delta^{(i)}) \subset \K[x_1, \dots, x_n]$ satisfy
    $$
        \ell(\F(\Delta^{(i)})) = n.
    $$
    In other words, facet ideals of skeletons of pure simplicial complexes have maximal analytic spread. 
\end{corollary}

\begin{proof}
    By \cref{p:purelevel}, the algebra $A(\Delta)$ is level. If we take the base field of $A(\Delta)$ to be of characteristic zero, \cref{hausel} implies the multiplication maps $\times L^{i - 1}: A(\Delta)_1 \to A(\Delta)_{i}$ are injective for every $i \leq d - 1$, where $L$ is the linear form given by the sum of variables. \cref{lefschetzmm} then implies the log matrix of $I_\Delta(0, i - 1) = \F(\Delta^{(i)})$ has rank $\dim A(\Delta)_1 = n$. \cref{rankanalyticspread} then implies $\ell(\F(\Delta^{(i)})) = n$ and $\ell(\F(\Delta^{(i)}))$ does not depend on the base field.
\end{proof}

\cref{c:maxanalyticspread} gives us the analytic spread of a large class of squarefree monomial ideals. Note that it is not possible to improve~\cref{c:maxanalyticspread} to also include the case $i = d$. In fact, if $I$ is the facet ideal of a pure simplicial complex with less facets than vertices (i.e $I$ has less generators than variables), the analytic spread is at most the number of generators of $I$, which is less than the number of variables. Due to several results relating analytic spread of an ideal and other algebraic invariants, \cref{c:maxanalyticspread} has consequences to the study of many algebraic invariants of $I$. We will explore such consequences in later sections.
 
\section{Linear type, SLP and the special case of simplicial forests}

In this section, we study the SLP of $A(\Delta)$ in characteristic zero, where  $\Delta$ is a $2$-dimensional complex. The results in this section use properties of Rees algebras of facet ideals, to show the Strong Lefschetz property of specific artinian reductions of Stanley-Reisner ideals.

\begin{proposition}\label{p:slp2dim}
    Let $\Delta$ be a pure $2$-dimensional simplicial complex with $f_0$ vertices, $f_2$ facets and assume the base field of $S$ has characteristic zero. If $f_2 \leq f_0$, then $A(\Delta)$ has the SLP if and only if $\ell(\F(\Delta)) = f_2$.
\end{proposition}

\begin{proof}
    Let $L = x_1 + \dots + x_n \in S_1$. By \Cref{lefschetzmm}, the multiplication map $\times L^2: A(\Delta)_1 \to A(\Delta)_3$ is represented by a constant multiple of $\log(\F(\Delta))$. Since $\Delta$ is pure and $2$-dimensional, every connected component of $\Delta(1)$ contains at least one triangle. We may then conclude $f_2 \leq f_0 \leq f_1$, where the second inequality follows since the $1$-skeleton of every connected component of $\Delta$ is a graph with at least one cycle (a triangle).
    
    If $A(\Delta)$ has the SLP, then the rank of $\log(\F(\Delta))$ is $f_2$ and by \cref{rankanalyticspread} we have 
    $$
        \ell(\F(\Delta)) = \rk \log(\F(\Delta)) = f_2.
    $$
    Conversely, if $\ell(\F(\Delta)) = f_2$, then $\rk \log(\F(\Delta)) = f_2$. Let $M_1, M_2$ be the matrices representing the maps $\times L: A(\Delta)_1 \to A(\Delta)_2$, $\times L: A(\Delta)_2 \to A(\Delta)_3$ respectively. Since  $2\log(\F(\Delta)) = M_2 M_1$, we conclude $f_2 = \rk \log(\F(\Delta)) \leq \min(\rk M_1, \rk M_2)$. By \Cref{t:daonair} and since every connected component of $\Delta(1)$ contains at least one triangle, we get $\rk M_1 = f_0$ and 
    $$
        f_2 = \rk M_2 M_1 \leq \min(\rk M_1, \rk M_2) = \min (f_0, \rk M_2) = \rk M_2,
    $$
    and in particular $f_2 \leq \rk M_2 \leq \min (f_1, f_2) = f_2$, which shows the SLP of $A(\Delta)$ in characteristic zero.
\end{proof}

As a direct corollary of \cref{p:slp2dim,l:lineartype} we get the following:

\begin{corollary}[\textbf{Linear type and SLP}]\label{c:lineartypeslp}
    Let $\Delta$ be a pure $2$-dimensional simplicial complex. If $\F(\Delta)$ is of linear type, then $A(\Delta)$ has the SLP in characteristic zero.    
\end{corollary}

Note that \Cref{c:lineartypeslp} gives us a connection between different properties of $\F(\Delta)$ and $\Nn(\Delta)$, where $\Delta$ is a $2$-dimensional complex. Namely information on the Rees algebra of $\F(\Delta) \subset S = \K[x_1, \dots, x_n]$ can be translated into information on whether the algebra 
$$
    S/(\Nn(\Delta), x_1^2, \dots, x_n^2)
$$
has the SLP.

In \cite{alisara}, the authors studied squarefree monomial ideals of linear type via facet ideals. By using \cref{c:lineartypeslp} and the main results from \cite{alisara}, we are able to find examples of $2$-dimensional simplicial complexes $\Delta$ such that $A(\Delta)$ satisfies the SLP. 

\begin{definition}
    Let $\Delta$ be a simplicial complex. We denote by L$(\Delta)$ the graph with vertices the facets of $\Delta$, and two facets $F_1, F_2$ are connected if $F_1 \cap F_2 \neq \emptyset$.
\end{definition}

\begin{theorem}[\cite{alisara}]\label{t:lineartype}
    Let $\Delta$ be a simplicial complex. If $L(\Delta)$ does not have even cycles, then $\F(\Delta)$ is of linear type. In particular, when $\Delta$ is a simplicial forest, $\F(\Delta)$ is of linear type.
\end{theorem}

\begin{remark}
    Note that in \cref{e:wslp}, the simplicial complex $\Delta$ is such that $L(\Delta)$ has an even cycle, and $A(\Delta)$ does not have SLP. In particular, as we saw in~\cref{e:notlineartype}, $\F(\Delta)$ is not of linear type.
\end{remark}

A direct consequence of \cref{t:lineartype,c:lineartypeslp} is that pure $2$-dimensional simplicial forests have the SLP. We can in fact prove a stronger statement. In order to show this result, we first show the structure of a few incidence ideals of simplicial forests.

\begin{proposition}\label{p:incidenceforest}
    Let $\Delta = \tuple{F_1, \dots, F_s}$ be a $d$ dimensional pure simplicial forest. Then the $(i, d)$-incidence complex of $\Delta$ is a simplicial forest for every $i < d$. 
\end{proposition}

\begin{proof}
    We will first show that a leaf of $\Delta$ corresponds to a leaf of $H_{\Delta}(i, d)$. Assume without loss of generality that $F_1$ is a leaf of $\Delta$ and $G \neq F_1$ is a facet of $\Delta$ such that $F_1 \cap F \subseteq F_1 \cap G$ for every facet $F \in \Delta$, $F \neq F_1$. By the definition of incidence ideals, for every facet $F$ of $\Delta$, there exists a unique facet $F'$ of $H_\Delta(i, d)$ such that the vertices of $F'$ are the $i$-faces of $F$. 
    
    If $|F_1 \cap F| < i + 1$ for every facet $F$ of $\Delta$, then $F_1'$ is an isolated simplex of $H_\Delta(i,d)$. If there exists a facet $F$ of $\Delta$ such that $|F_1 \cap F| \geq i + 1$, then since $F_1 \cap F \subseteq F_1 \cap G$, we conclude every subset of size $i + 1$ of $F_1 \cap F$ is also a subset of size $i + 1$ of $G$. Subsets of size $i + 1$ are exactly the $i$-faces of a simplex, in particular 
    $$
        F_1' \cap F' \subseteq F_1' \cap G'
    $$
    and hence $F_1'$ is a leaf of $H_\Delta(i, d)$. 

    Next, consider the complex $\Delta \setminus F_j = \tuple{F_1, \dots, F_{j - 1}, F_{j + 1}, \dots, F_s}$. The $(i, d)$-incidence complex of $\Delta \setminus F_j$ is $H_{\Delta \setminus F_j}(i, d) = \tuple{F_1', \dots, F_{j - 1}', F_{j + 1}', \dots, F_s'}$, so every connected component of every subcomplex $H_\Delta(i, d)' = \tuple{F_{j_1}', \dots, F_{j_r}'}$ of $H_\Delta(i, d)$ has a leaf, which implies $H_\Delta(i, d)$ is a simplicial forest. 
\end{proof}

\begin{theorem}[\textbf{SLP in positive characteristics of $2$-dimensional simplicial forests}]\label{t:forestslp}
    Let $\Delta$ be a simplicial forest where every connected component of $\Delta$ is a $2$-dimensional simplicial tree. Then $A(\Delta)$ has the SLP in every characteristic $\neq 2$.
\end{theorem}

\begin{proof}
    Since every connected component is $2$-dimensional, every connected component of $\Delta(1)$ contains a triangle. By \cref{t:daonair}, the map $\times L: A(\Delta)_1 \to A(\Delta)_2$ has full rank in every odd characteristic, where $L = x_1 + \dots + x_n$. When $\Delta$ is pure, \cref{c:lineartypeslp,t:lineartype} imply the maps $\times L: A(\Delta)_2 \to A(\Delta)_3$ and $\times L^2: A(\Delta)_1 \to A(\Delta)_3$ have full rank in characteristic zero. \cref{p:forestgcdminors} then imply the maps have full rank in every characteristic. When $\Delta$ is not pure, let $\bar \Delta = \tuple{T_1, \dots, T_s}$ where $T_i$ are exactly the facets of $\Delta$ that are triangles. The maps $\times L^2: A(\Delta)_1 \to A(\Delta)_3$ and $\times L^2: A(\bar \Delta)_1 \to A(\bar \Delta)_3$, where $L$ is taken to be the sum of the variables on the corresponding polynomial ring, are represented by the matrices $M_1, M_2$ respectively. The only difference between $M_1$ and $M_2$ is that $M_1$ might have extra columns where every entry is zero (if such a column exists, it corresponds to a vertex of $\Delta$ that is not in any triangle of $\Delta$), and so they have the same rank by \cref{l:lessfacetsforests}. Since $\bar \Delta$ is a pure simplicial forest, we conclude $\times L^2: A(\bar \Delta)_1 \to A(\bar \Delta)_3$ has full rank in every characteristic by the arguments above. From the fact that $\bar \Delta$ is a simplicial forest, we know 
    $$
        \dim A(\Delta)_1 \geq \dim A(\bar \Delta)_1 \geq \dim A(\bar \Delta)_3 = \dim A(\Delta)_3 = \rk M_1 = \rk M_2
    $$
    and so $M_1$ also has full rank. Similarly, the matrix that represents the map $\times L: A(\Delta)_2 \to A(\Delta)_3$ is the same matrix that represents the map $\times L: A(\bar \Delta)_2 \to A(\bar \Delta)_3$, except it has one zero column for each facet of $\Delta$ that is an edge. In particular, since $H_\Delta(1,2)$ is a simplicial forest where each facet is either an isolated vertex, or a triangle, we conclude $\dim A(\Delta)_2 \geq \dim A(\Delta)_3$ and so adding the extra zero columns in $\times L: A(\Delta)_2 \to A(\Delta)_3$ does not change the rank in any characteristic.
\end{proof}

\begin{remark}
    Let $X = \{i + 1 \st \dim A(\Delta)_i \leq \dim A(\Delta)_{i + 1}\}$, where $\Delta$ is a simplicial complex. \cref{t:forestslp} gives us a family of complexes $\Delta$ such that $A(\Delta)$ satisfies the SLP in every characteristic $p$ such that $p$ does not divide the least common multiple of $X$. Since the facet ideals of the complexes in this family are of linear type, \cref{t:forestslp} gives a partial answer to \cite[Question 62]{lefschetzmm}.
\end{remark}

\section{Symbolic powers and Lefschetz properties}\label{s:symbolic}

We now turn our attention to the study of symbolic powers and Lefschetz properties of squarefree monomial ideals.
In this section, we focus on the connections between the Lefschetz properties of $A(\Delta)$ and the symbolic powers of $\F(\Delta^{(i)})$. Throughout this section, $S$ will always be a polynomial ring over a field $\K$ of characteristic zero.

\begin{definition}
    Let $I$ be an ideal of $S$. The \textbf{$m$-th symbolic power} of $I$, denoted $I^{(m)}$, is defined as 
    $$
        I^{(m)} = \cap_{P \in \text{Ass(I)}} (I^m S_P \cap S),
    $$
    where $S_P$ denotes localization at the prime $P$.

    When $I$ is a radical ideal (in particular, when $I$ is a squarefree monomial ideal), the following holds for every $m$
    $$
    I^{(m)} = \cap_{P \in \text{Ass}(I)} P^m. 
    $$    
\end{definition}

Although in general $I^{(m)} \supset I^m$ always holds, the converse fails for example when 
$$
    I = (xy, xz, yz) \subset \K[x,y,z],
$$
since $xyz \in I^{(2)}$, but $xyz \not \in I^2$.

When a squarefree monomial ideal satisfies $I^{(m)} = I^m$ for every $m$, we say the ideal is \textbf{normally torsion-free}. The class of normally torsion-free squarefree monomial ideals can be characterized in terms of hypergraph and integer programming properties. Finding classes of squarefree monomial ideals that are normally torsion-free is an important problem in Combinatorial Commutative Algebra. Such classes include for example edge ideals of bipartite graphs and facet ideals of simplicial forests. We refer the reader to \cite{MR2932582} for more details. One particular result that we will need is the following:

\begin{proposition}\label{p:ntfanalyticspread}\cite{MR2932582}
    If $I \subset S$ is a squarefree monomial ideal such that $\ell(I) = \dim S$, then $I$ is not normally torsion-free.
\end{proposition}

Our next goal is to construct families of level monomial algebras failing SLP. Here we focus on the well known construction of "whiskering", but similar arguments can also be made for the "grafting" construction (see~\cites{whiskered2,faridiforest} for the definition). For the exact analogue statements on grafted complexes see~\cite{phd}.

\cref{p:ntfanalyticspread} implies the following general relation between SLP and symbolic powers:

\begin{theorem}[\textbf{SLP and symbolic powers are not compatible}]\label{t:slpsymbolic}
    Let $\Delta$ be a pure simplicial complex with at least as many facets as vertices. If $\F(\Delta)$ is normally torsion-free, $A(\Delta)$ fails the SLP in characteristic zero.
\end{theorem}

\begin{proof}
    Let $d$ denote the dimension of $\Delta$. Since $\Delta$ is pure and has as many facets as vertices, $\F(\Delta)$ is a squarefree monomial ideal generated in the same degree. \cref{lefschetzmm} then implies the rank of the multiplication map $\times L^{d}: A(\Delta)_1 \to A(\Delta)_{d + 1}$, where $L = x_1 + \dots + x_n$ is $\ell(\F(\Delta))$. Since $\Delta$ has at least as many facets as vertices, the map $\times L^d: A(\Delta)_1 \to A(\Delta)_{d + 1}$ has full rank if and only if it is injective, and in particular $\rk \times L^d = \dim A(\Delta)_1 = \dim S$. The result then follows by \cref{p:ntfanalyticspread}.
\end{proof}

\begin{remark}\label{r:atleast}
    Note that \cref{t:slpsymbolic} only holds when $\Delta$ has at least as many facets as vertices. Pure simplicial forests of dimension $2$ are examples of simplicial complexes with less facets than vertices that are both normally torsion-free and satisfy the SLP in characteristic zero.
\end{remark}

\cref{t:slpsymbolic} allows us to construct families of simplicial complexes $\Delta$ such that $A(\Delta)$ fails the SLP.

\begin{definition}
    Let $\Delta$ be a simplicial complex and 
    $$
        \F(\Delta) = (x_{i_{1,1}}, \dots, x_{i_{1,r_1}}) \cap \dots \cap (x_{i_{s,1}}, \dots, x_{i_{s,r_s}}) \subset S = \K[x_1, \dots, x_n]
    $$
    the primary decomposition of its facet ideal. The ideal
    $$
        J(\Delta) = (x_{i_{1,1}}\dots x_{i_{1,r_1}}, \dots, x_{i_{s,1}}\dots x_{i_{s,r_s}}) \subset S
    $$
    is called the \textbf{cover ideal} of $\Delta$.
\end{definition}

\begin{proposition}[Corollary 4.38~\cite{MR2932582}]\label{p:dualntf}
        Let $H$ be a unimodular $r$-hypergraph and $\Delta$ the corresponding simplicial complex. Then $\F(\Delta)$ and $J(\Delta)$ are normally torsion-free.
\end{proposition}

\begin{lemma}\label{l:dualntf}
    The cover ideals of bipartite graphs are normally torsion-free.
\end{lemma}

\begin{proof}
    By \cref{t:unimodular}, we know the associated hypergraphs are unimodular and every edge has the same size. The result then follows directly by \cref{p:dualntf}. 
\end{proof}

\cref{t:unimodular,p:dualntf} are the tools we use to find families of complexes failing the SLP. To exemplify this approach we first do an example.

\begin{example}\label{e:coverpath}
    Let $G$ be the graph below and 
    $$
        \F(G) = (a b, b c, c d, a e, b f, c g, d h) \subset \K[a,\dots, h]
    $$
    be its facet ideal, where $\K$ is a field of characteristic zero.

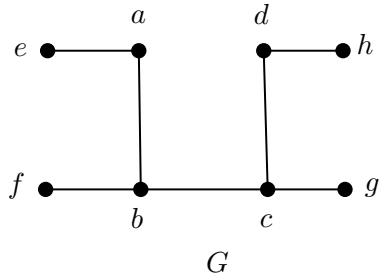
\begin{center}
\tikzset{every picture/.style={line width=0.75pt}} 

\begin{tikzpicture}[x=0.75pt,y=0.75pt,yscale=-1,xscale=1]

\draw    (137.5,90) -- (138.5,160) ;
\draw    (138.5,160) -- (202.5,160) ;
\draw    (200.5,90) -- (202.5,160) ;
\draw    (137.5,90) -- (91.5,90) ;
\draw [shift={(91.5,90)}, rotate = 180] [color={rgb, 255:red, 0; green, 0; blue, 0 }  ][fill={rgb, 255:red, 0; green, 0; blue, 0 }  ][line width=0.75]      (0, 0) circle [x radius= 3.35, y radius= 3.35]   ;
\draw [shift={(137.5,90)}, rotate = 180] [color={rgb, 255:red, 0; green, 0; blue, 0 }  ][fill={rgb, 255:red, 0; green, 0; blue, 0 }  ][line width=0.75]      (0, 0) circle [x radius= 3.35, y radius= 3.35]   ;
\draw    (200.5,90) -- (240.5,90) ;
\draw [shift={(240.5,90)}, rotate = 0] [color={rgb, 255:red, 0; green, 0; blue, 0 }  ][fill={rgb, 255:red, 0; green, 0; blue, 0 }  ][line width=0.75]      (0, 0) circle [x radius= 3.35, y radius= 3.35]   ;
\draw [shift={(200.5,90)}, rotate = 0] [color={rgb, 255:red, 0; green, 0; blue, 0 }  ][fill={rgb, 255:red, 0; green, 0; blue, 0 }  ][line width=0.75]      (0, 0) circle [x radius= 3.35, y radius= 3.35]   ;
\draw    (138.5,160) -- (90.5,160) ;
\draw [shift={(90.5,160)}, rotate = 180] [color={rgb, 255:red, 0; green, 0; blue, 0 }  ][fill={rgb, 255:red, 0; green, 0; blue, 0 }  ][line width=0.75]      (0, 0) circle [x radius= 3.35, y radius= 3.35]   ;
\draw [shift={(138.5,160)}, rotate = 180] [color={rgb, 255:red, 0; green, 0; blue, 0 }  ][fill={rgb, 255:red, 0; green, 0; blue, 0 }  ][line width=0.75]      (0, 0) circle [x radius= 3.35, y radius= 3.35]   ;
\draw    (202.5,160) -- (241.5,160) ;
\draw [shift={(241.5,160)}, rotate = 0] [color={rgb, 255:red, 0; green, 0; blue, 0 }  ][fill={rgb, 255:red, 0; green, 0; blue, 0 }  ][line width=0.75]      (0, 0) circle [x radius= 3.35, y radius= 3.35]   ;
\draw [shift={(202.5,160)}, rotate = 0] [color={rgb, 255:red, 0; green, 0; blue, 0 }  ][fill={rgb, 255:red, 0; green, 0; blue, 0 }  ][line width=0.75]      (0, 0) circle [x radius= 3.35, y radius= 3.35]   ;

\draw (132,68) node [anchor=north west][inner sep=0.75pt]   [align=left] {$\displaystyle a$};
\draw (132,168) node [anchor=north west][inner sep=0.75pt]   [align=left] {$\displaystyle b$};
\draw (197,171) node [anchor=north west][inner sep=0.75pt]   [align=left] {$\displaystyle c$};
\draw (73,85) node [anchor=north west][inner sep=0.75pt]   [align=left] {$\displaystyle e$};
\draw (194,65) node [anchor=north west][inner sep=0.75pt]   [align=left] {$\displaystyle d$};
\draw (70,150) node [anchor=north west][inner sep=0.75pt]   [align=left] {$\displaystyle f$};
\draw (250,153) node [anchor=north west][inner sep=0.75pt]   [align=left] {$\displaystyle g$};
\draw (246,80) node [anchor=north west][inner sep=0.75pt]   [align=left] {$\displaystyle h$};
\draw (170,190) node [anchor=north west][inner sep=0.75pt]   [align=left] {$\displaystyle G$};
\end{tikzpicture}
\end{center}

Since $G$ is bipartite, we know it is unimodular as a hypergraph. Moreover, the cover ideal of $G$ is given by
$$
    J(G) = (dcba, hcba, gdba, fdca, hfca, edcb, hecb, gedb).
$$
By \cref{t:unimodular,l:dualntf} $J(G)$ is normally torsion-free. In particular,
$$
    \Delta = \tuple{dcba, hcba, gdba, fdca, hfca, edcb, hecb, gedb}
$$
is a pure simplicial complex with a normally torsion-free facet ideal that has as many facets as vertices. By \cref{t:slpsymbolic} $A(\Delta)$ fails the SLP in characteristic zero.
\end{example}

The example above uses the construction of \say{whiskering} a graph. The connections of this construction and the Lefschetz properties of algebras of the form $A(\Delta)$ were first studied in \cite{cooper2023weak}. Here, we show how this construction can be used to give us simplicial complexes that fail the SLP.

\begin{definition}
    Let $G$ be a graph with vertex set $V(G) = \{v_1, \dots, v_n\}$. The \textbf{whiskered graph} $w(G)$ has vertex set $V(w(G)) = V(G) \cup \{w_1, \dots, w_{n}\}$ and edge set $E(w(G)) = \{v_i w_i \st i \in [n]\} \cup E(G)$.
\end{definition}

The graph in \cref{e:coverpath} is $w(P_4)$, where $P_4$ is the path graph on $4$ vertices. The following lemma allows us to generalize \cref{e:coverpath}.

\begin{lemma}[Theorem 5~\cite{MR3104513}]\label{l:faceswhisker}
    If $G$ is a whiskered graph on $2n$ vertices, then the Stanley-Reisner complex of $\F(w(G))$ is a pure simplicial complex of dimension $n - 1$. Moreover, facets of the Stanley-Reisner complex of $\F(w(G))$ are in bijection with independent sets of $G$. 
    
    In particular, generators of $J(w(G))$ are in bijection with independent sets of $G$.
\end{lemma}

\begin{proof}
    The first part of the statement is exactly Theorem 5 from \cite{MR3104513}. The last part of the statement follows since $J(w(G))$ is generated by the following set 
    $$ 
        \Big{\{}\prod_{i \not\in F} x_i \st F \text{ is a facet of the Stanley-Reisner complex of $\F(w(G))$}\Big{\}}.
    $$
\end{proof}

A direct consequence of \cref{l:faceswhisker} is the fact that $J(w(G))$ is generated by monomials of degree $n$ for any graph $G$ on $n$ vertices. 

\begin{lemma}\label{l:maxedges}
    A bipartite graph on $n$ vertices has at most $\frac{n^2}{4}$ edges.
\end{lemma}

The previous results in this section then imply the following.
 
\begin{proposition}\label{p:unimodularslpbad}
    If $\Delta$ is a pure simplicial complex on vertex set $[n]$ that is unimodular as a hypergraph, $\F(\Delta) \subset S$ is unmixed and $J(\Delta)$ is generated by at least $n$ elements. Then $A(\Gamma)$ fails the SLP in characteristic zero, where $\Gamma$ is the facet complex of $J(\Delta)$.
\end{proposition}

\begin{proof}
    By \cref{p:dualntf}, we know both $\F(\Delta)$ and $J(\Delta)$ are normally-torsion free. Since $\F(\Delta)$ is unmixed, the facet complex $\Gamma$ of $J(\Delta)$ is pure, and has at least as many facets as vertices. The result then follows by \cref{t:slpsymbolic}.
\end{proof}

\begin{corollary}[\textbf{Facet complexes of cover ideals of whiskered bipartite graphs fail the SLP}]\label{c:bipartitewhisker} 
    Let $G$ be a bipartite graph on $n \geq 5$ vertices, let 
    $$
        J(w(G)) = (x_{i_{1,1}}\dots x_{i_{1,r_1}}, \dots, x_{i_{s,1}} \dots x_{i_{s,r_s}}) \subset \K[x_1, \dots, x_{2n}] 
    $$
    and
    $$
    \Delta = \tuple{\{i_{1,1}, \dots, i_{1, r_1}\}, \dots, \{i_{s,1}, \dots, i_{s,r_s}\}}.
    $$
    Then $A(\Delta)$ is a level monomial algebra that fails the SLP.    
\end{corollary}

\begin{proof}
    We want to show that for $n \geq 5$, the Stanley-Reisner complex of $\F(G)$ has at least $2n$ faces. By \cref{l:maxedges}, the number of edges in the Stanley-Reisner complex of $\F(G)$ is at least 
    $$
        \binom{n}{2} - \frac{n^2}{4} = \frac{n^2 - 2n}{4}
    $$
    and since the number of edges must be an integer, we conclude the Stanley-Reisner complex of $\F(G)$ has at least $n - 1$ edges for $n \geq 5$. By \cref{l:faceswhisker}, the generators of $J(w(G))$ are in bijection with independent sets of $G$ (including the empty set), thus we have the following
    \begin{enumerate}
        \item The empty set is an independent set of $G$
        \item Every vertex is an independent set of $G$
        \item Every edge of the Stanley-Reisner complex is an independent set of $G$.
    \end{enumerate}
    Counting the sets above we see that we have at least $1 + n + n - 1 = 2n$ independent sets, and so $J(w(G))$ is generated by at least $2n$ monomials. 
    Finally, since $w(G)$ is a bipartite graph, we know $\F(w(G))$ is an unmixed ideal and $J(w(G)) = \F(\Delta)$ is a normally torsion-free ideal where $\Delta$ is a pure simplicial complex with at least as many facets as vertices. The result then follows by \cref{p:unimodularslpbad}.
\end{proof}

\section{Symbolic defect polynomials}\label{s:symbolicdefect}

In \cref{s:symbolic}, we studied ideals whose symbolic powers are exactly the ordinary powers, but as was mentioned before, this is not always the case.
In order to measure how different the symbolic powers of an ideal $I$ are from its ordinary powers, the authors in \cite{MR3906569} introduced the following numerical invariants of an ideal $I$:

\begin{definition}
    Let $I$ be an ideal. The \textbf{$m$-th symbolic defect} of $I$, denoted by $\sdefect(I, m)$ is
    $$
        \sdefect(I, m) = \mu(I^{(m)}/I^m)
    $$
    where $\mu(M)$ denotes the minimal number of generators of $M$.
\end{definition}
 
Most of the results on the symbolic powers of squarefree monomial ideals, take the approach of associating hypergraphs and graphs to squarefree monomial ideals. Here, we focus on the combinatorial properties of simplicial complexes that are connected to symbolic powers. 
We now briefly describe how the results from previous sections lead us to \say{symbolic defect polynomials} (see \cref{d:symbolicpol}). \cref{hausel} shows that in characteristic zero, maps of the form $\times L^i \colon A(\Delta)_1 \to A(\Delta)_{i + 1}$ are injective for $i \leq d - 1$ when $\Delta$ is a pure $d$-dimensional simplicial complex. \cref{lefschetzmm} then says the analytic spread of $\F(\Delta^{(i)}) \subset S$ is always equal to the dimension of $S$. 
Finally, \cref{p:ntfanalyticspread} implies facet ideals of skeletons of pure complexes are not normally torsion-free. We have shown the following statement.

\begin{proposition}\label{p:symbolicmotivation}
    For every $1 \leq i \leq d - 1$ there exists $j$ such that $F(\Delta^{(i)})^{(j)} \neq \F(\Delta^{(i)})^j$.     
\end{proposition}

This leads us to the study of sequences of the form
\begin{equation}\label{eq:1}
    \sdefect(\F(\Delta^{(1)}), j), \dots, \sdefect(\F(\Delta^{(d-1)}), j),   
\end{equation}
    where $\Delta$ is a pure $d$-dimensional complex. 
    In the original paper \cite{MR3906569}, the authors suggested the study of sequences $\sdefect(I, j)$ where $I$ is fixed and $j$ varies. Results from this perspective on the symbolic defect of (squarefree) monomial ideals can be found for example in~\cites{MR3986423,MR4556315,symbolicdefectedgeideals}. 
    As we will see, the sequence in \eqref{eq:1} contains information on the $f$-vector of $\Delta$. We begin by stating the following result from~\cite{MR3906569} on the symbolic defect of \textit{star configurations} in our setting.

\begin{theorem}\cite{MR3906569}\label{t:sdefectsimplex}
    Let $I_{n,l} \subset S = \K[x_1, \dots, x_n]$ be the ideal generated by every squarefree monomial of degree $l$ in $S$. In other words,
    $$
        I_{n,l} = \bigcap_{1 \leq i_1 < \dots < i_{n - l + 1} \leq n} (x_{i_1}, \dots, x_{i_{n - l + 1}}).
    $$
    Then $\sdefect(I_{n,l}, 2) = \binom{n}{l + 1}$.
\end{theorem}

Note that the ideals $I_{n, l}$ defined above can be seen as $\F(\Delta(l-1))$, where $\Delta$ is the $n-1$ dimensional simplex. 
Since $\binom{n}{l + 1}$ is the number of $l$-faces of the $n-1$ dimensional simplex, \cref{t:sdefectsimplex} says the second symbolic defect of the facet ideals of skeletons of the simplex can be understood from the $f$-vector of the simplex. 
This observation, together with \cref{p:symbolicmotivation} motivates the definition of \textit{symbolic defect polynomials}.

\begin{definition}\label{d:symbolicpol}
    Let $\Delta$ be a $d$-dimensional simplicial complex. The \textbf{$m$-th symbolic defect polynomial} of $\Delta$ is
    $$
        \mu_m(\Delta, t) = t^2\sum_{i = 1}^{d - 1} \sdefect(\F(\Delta^{[i]}), m)t^i 
    $$
    where $\Delta^{[i]}$ is the simplicial complex whose facets are the $i$-faces of $\Delta$. Note that when $\Delta$ is pure, $\Delta^{[i]}$ is equal to the $i$-th skeleton $\Delta^{(i)}$ of $\Delta$.
\end{definition}

\cref{t:sdefectsimplex} implies the second symbolic polynomial of an $n-1$-dimensional simplex $\Delta$ is equal to 
$$
    \mu_2(\Delta, t) = (1 + t)^n - 1 - nt- \binom{n}{2}t^2 = t^2\sum_{i = 1}^{n - 1} \binom{n}{i + 2}t^i.
$$
\begin{corollary}
    Let $G$ be a star graph on $n + 1$ vertices, and $\Delta$ the Stanley-Reisner complex of $\F(G)$. Then
    $$
        \mu_2(\Delta, t) = t^2\sum_{i = 1}^{n - 2} \binom{n}{i + 2}t^i = t^2 \sum_{i  = 1}^{n - 2} f_{i + 1} t^i.
    $$
\end{corollary}

\begin{proof}
    The Stanley-Reisner complex of a star graph is a simplex with an isolated vertex. Since the isolated vertex does not change the symbolic defect polynomial, the result holds directly by \cref{t:sdefectsimplex}.
\end{proof}

\begin{example}\label{e:symbolicdefectpolynomial}
    Let $G = w(P_4)$ be the whiskered path graph from \cref{e:coverpath} and 
    $$
        \Delta = \tuple{efgh, afgh, begh, cefh, acfh, defg, adfg, bdeg}
    $$
    the Stanley-Reisner complex of $\F(w(P_4))$. Then the second and third symbolic polynomials of $\Delta$ are
    $$
        \mu_2(\Delta, t) = 22t^3 + 8t^4 \qand \mu_3(\Delta, t) = 184t^3 + 106t^4.
    $$
    Moreover, the $f$-vector of $\Delta$ is $(1,8,21,22,8)$. Note that the coefficients of $\mu_2(\Delta, t)$ are entries of the $f$-vector of $\Delta$.
\end{example}

The facet ideals of the skeletons of a simplicial complex $\Delta$ are invariant under deleting higher dimensional faces. In other words, if $\Delta$ contains the boundary of an $n$-dimensional simplex $\tau$, but not $\tau$ itself, then $\F(\Delta(n - 1)) = \F(\Delta'(n - 1))$, where the faces of $\Delta'$ are the faces of $\Delta$, and $\tau$. For this reason, we focus on simplicial complexes where the boundary of a simplex $\tau$ is a subcomplex of $\Delta$ if and only if $\tau$ is a face of $\Delta$. These complexes are called \textbf{flag complexes}, and they can be characterized as the Stanley-Reisner complexes of squarefree monomial ideals generated in degree $2$.

When $G$ is a graph, then $\sdefect(\F(G), 2)$ has a simple description in terms of the combinatorial structure of $G$. Moreover, this description gives us even more evidence on the connections between symbolic defect polynomials and $f$-vectors. Before stating the result, we set some notation from Graph theory.

\begin{definition}
    Let $G$ be a graph with vertex set $V = [n]$ and edge set $E(G)$. 
    Given $v \in V$, we denote by $N[v]$ the \textbf{closed neighborhood} of $v$, that is, $N[v] = \{u : u \in V, \{u, v\} \in E(G)\} \cup \{v\}$. Similarly, the closed neighborhood of a set $A \subset V$ is $N[A] = \cup_{v \in A} N[v]$. We write $G[A]$ for the graph with vertex set $A$ and edge set $E' = \{\{u, v\} : \{u, v\} \in E(G), u,v \in A\}$, which is called the \textbf{induced subgraph by $A$}. 
    A set of vertices $A \subset V(G)$ is said to be an \textbf{independent set} if there is no edge between elements of $A$.
    Given a monomial $m$, we denote by $\supp(m) = \{i \st x_i \mid m\}$ the \textbf{support} of $m$. Moreover, we write $G_m$ for the induced subgraph $G[\supp(m)]$.

    We denote by $K_n$ the complete graph on $n$ vertices.
\end{definition}

\begin{theorem}[\cite{kumarsymbolicdefect,symbolicdefectedgeideals}]\label{t:triangle}
    Let $G$ be a graph. Then $\F(G)^{(2)} = \F(G)^2 + (x_i x_j x_k \st G[\{i,j,k\}] = K_3)$, in particular, $\sdefect(\F(G), 2)$ is the number of triangles of $G$.
\end{theorem}

Since we have a description of the generators of $\F(G)^{(2)}$ for any graph $G$, we can completely describe the Hilbert function of $\F(G)^{(2)}/\F(G)^2$. Before stating the next corollary, we mention the Stanley-Reisner complex of a graph $G$ is called the \textbf{independence complex of $G$}, since its faces are independent sets of $G$. For this reason, we denote by $\Ind(G)$ the Stanley-Reisner complex of $\F(G)$.

\begin{corollary}\label{c:hilbert}
    Let $G$ be a graph and $\mathcal{T}$ the set of triangles of $G$. The Hilbert function of $M = \F(G)^{(2)}/\F(G)^2$ is
    $$
        H_M(z) = \sum_{T \in \mathcal{T}} H_{S/I_T}(z - 3)
    $$
    where $I_T = \F(G - N[T])$ and $H_M(i)$ denotes the Hilbert function of $M$ at $i$.
\end{corollary}

\begin{proof}
    By \cref{t:triangle}, we know $\F(G)^{(2)} = \F(G)^2 + (x_i x_j x_k \st \{i,j,k\} \in \mathcal{T})$. For $T = \{i,j,k\} \in \mathcal{T}$, set $x_T = x_i x_j x_k$. Note that if a monomial $h$ is nonzero in $\F(G)^{(2)}/\F(G)^2$, then $G_h$ is a single triangle $T$ with extra isolated vertices, in particular, $\supp(h) \subset V(G - N[T]) \cup T$. The following equalities hold 
    
    \begin{align*}
        H_M(z)  & = \dim_{\K} \Bigg{(}\frac{\F(G)^2 + (x_i x_j x_k \st \{i,j,k\} \in \mathcal{T})}{\F(G)^2}\Bigg{)}_z \\
                & = \Big{|}\bigcup_{T \in \mathcal{T}} \big{\{} m x_T \st \supp(m) \in \Ind(G - N[T]) \qand \deg mx_T = z \big{\}}\Big{|} \\
                & = \Big{|}\bigsqcup_{T \in \mathcal{T}} \big{\{} m \st \supp(m) \in \Ind(G - N[T]) \qand \deg m = z - 3 \big{\}}\Big{|} \\
                & = \sum_{T \in \mathcal{T}} H_{S/I_T}(z - 3),
    \end{align*}
    where $\sqcup$ denotes disjoint union.
\end{proof}

\begin{example}
    Let $G = \Delta^{(1)}$, where $\Delta$ is the Stanley-Reisner complex of $w(P_4)$ from \cref{e:stanleyfacet}. Then the set $\mathcal{T}$ from \cref{c:hilbert} is 
    $$
        \mathcal{T} = \{\{a,e,f\}, \{a,e,c\}, \{c,d,e\}, \{d,e,f\}, \{b,d,f\}\}.
    $$
    The graphs $G - N[T]$ for $T \in \mathcal{T}$ are the empty graph, except for $G - N[\{a,e,c\}]$, which is a single isolated vertex $b$. In particular, the Hilbert function of $M = \F(G)^{(2)}/\F(G)^2$ is 
    $$
        H_M(3) = |\{aef, aec, cde, def, bdf\}| = 5 \qand H_M(i) = |\{aec b^{i - 3}\}| = 1 \text{ for $i > 3$.} 
    $$
\end{example}

Given a pure $d$-dimensional simplicial complex $\Delta$, \cref{hausel,p:ntfanalyticspread} imply for every $i < d$, there exists $m$ such that $\F(\Delta^{(i)})^{(m)} \neq \F(\Delta^{(i)})^{m}$. Our final goal is to show that $m$ can always be taken to be $2$, and for the context of symbolic defect polynomials, $\Delta$ does not have to be pure.

\begin{theorem}\label{t:nozero}
    The coefficients of the second symbolic defect polynomial of a $d$-dimensional simplicial complex $\Delta$ form a sequence with no internal zeros. Moreover, if $\mu_2(\Delta, t) = c_2 t^3 + \dots + c_d t^{d + 1}$, then $c_i \geq f_i(\Delta)$ for $i > 2$ and $c_2 = f_2$..
\end{theorem}

\begin{proof}
    We will show that $\F(\Delta^{[i]})^{(2)} \neq \F(\Delta^{[i]})^2$ for every $i < d$, where $\Delta^{[i]}$ is the simplicial complex whose facets are the $i$-faces of $\Delta$. 
    Since $i < d$, we know for every $i$, the ideal $\F(\Delta^{[i]})$ is the facet ideal of a simplicial complex that contains the boundary of an $i + 1$-simplex, otherwise the dimension of $\Delta$ would be less than $d$. 
    Let $\{j_1, \dots, j_{i + 2}\}$ be the vertices of such a simplex and $\Gamma_i$ the Stanley-Reisner complex of $\F(\Delta^{[i]})$. We know the facets of $\Gamma_i$ correspond to the complements of the associated primes of $\F(\Delta^{[i]})$. Note that since the facet complex of $\F(\Delta^{[i]})$ contains the boundary of the simplex $\{j_1, \dots, j_{i + 2}\}$, every subset of size $i + 1$ of this set of indices is a generator of $\F(\Delta^{[i]})$. In particular, since generators of $\F(\Delta^{[i]})$ correspond to minimal non-faces of $\Gamma_i$, we conclude every facet of $\Gamma_i$ contains at most $i$ elements of $\{j_1, \dots, j_{i + 2}\}$, and thus for every facet $F$ of $\Gamma_i$, there exists $k_1, k_2 \in \{j_1, \dots, j_{i + 2}\} \cap ([n] \setminus F)$ so $m = x_{j_1} \dots x_{j_{i + 2}} \in (x_l : l \in [n] \setminus F)^2$ for every facet $F$, and in particular, $m \in \F(\Delta^{[i]})^{(2)}$.
    Since $\deg m = i + 2 < 2(i + 1)$, we conclude $m \not \in \F(\Delta_i)^2$, and thus $\mu(\F(\Delta_i)^{(2)}/\F(\Delta_i)^2) > 0$.
    
    For the last part of the statement, the $i = 2$ case follows from \cref{t:triangle}. For the $i > 2$ case, note that $\F(\Delta^{[i]})$ is generated by monomials of degree $i + 1$, and we have found elements of degree $i + 2$ that are nonzero in $\F(\Delta^{[i]})^{(2)}/\F(\Delta^{[i]})^2$. 
    Since $\F(\Delta^{[i]})^{(2)} \subset (x_1, \dots, x_n) \F(\Delta^{[i]})$ (see \cite[Proposition 3.8]{MR3589840}), we conclude these elements which correspond to $i+1$-faces of $\Delta$ must be generators of $\F(\Delta^{[i]})^{(2)}$, hence $\sdefect(\F(\Delta^{[i]}), 2) = c_i \geq f_i > 0$.
\end{proof}

\section{(Mixed) Eulerian numbers, (mixed) multiplicities and pure simplicial complexes}\label{s:mvmm}

We now focus our attention to the study of multiplicities of squarefree monomial ideals. We refer the reader to \cite{MR2320648,MR1812818} for more details on mixed multiplicities of ideals and mixed volumes of polytopes.  Given a sequence of ideals $I_1, \dots, I_s \subset S = \K[x_1, \dots, x_n]$ and $\m$ the maximal ideal $(x_1, \dots, x_n)$, the multigraded Hilbert function of the algebra
$$
    R(\m | I_1, \dots, I_s) = \bigoplus_{\textbf{u} \in \N^{s + 1}} \frac{\m^{u_0}I_1^{u_1} \dots I_s^{u_s}}{\m^{u_0 + 1}I_1^{u_1} \dots I_s^{u_s}}
$$
is a polynomial $P$ for $u_0, \dots, u_s \gg 0$. By the results in \cite{MR1812818}, the polynomial $P(u_0, \dots, u_{s})$ can be written as
$$
    P(u_0, \dots, u_{s}) = \sum_{a_0 + \dots + a_s = n - 1} \frac{e_{a_0, \dots, a_s}(\m | I_1, \dots, I_s)}{a_0! \dots a_s!} u_0^{a_0} \dots u_s^{a_s} + \text{ lower degree terms}.
$$
The numbers $e_{\textbf{a}}(\m | I_1, \dots, I_s) = e_{a_0, \dots, a_s}(\m | I_1, \dots, I_s)$ are called the \textbf{mixed multiplicities} of $\m, I_1, \dots, I_s$. The following results are the key for the results in this section.

\begin{theorem}[Theorem 4.4 \cite{MR4157582}]\label{t:positivity}
    Let $S = \K[x_1, \dots, x_n]$, $\m = (x_1, \dots, x_n)$ and $I_1, \dots, I_s$ be arbitrary ideals of $S$ such that $I_i$ is generated by homogeneous elements of same degree $d_i > 0$. Let ${\bf a} = (a_0, \dots, a_s) \in \N^{s + 1}$ such that $a_0 + \dots + a_s = n - 1$. Then $e_{{\bf a}}(\m | I_1, \dots, I_s) > 0$ if and only if for every $J = \{j_1, \dots, j_t\} \subset \{1, \dots s\}$ the inequality
    $$
        a_{j_1} + \dots + a_{j_t} \leq \ell(I_{j_1}\dots I_{j_t}) - 1
    $$
    holds, where $\ell(I_{j_1}\dots I_{j_t})$ is the analytic spread of the ideal $I_{j_1}\dots I_{j_t}$.
\end{theorem}

\begin{lemma}[Corollary 3.12 \cite{MR4413366}]\label{l:spreadproduct}
    Let $I, J$ be nonzero ideals of $S = \K[x_1, \dots, x_n]$. Then $\ell(IJ) \geq \max \{\ell(I), \ell(J)\}$.
\end{lemma}

We can then show the following regarding positivity of mixed multiplicities of facet ideals of skeleta.

\begin{theorem}[\bf{Positivity of mixed multiplicities of skeleta of pure complexes}]\label{t:posmm}
    Let $S = \K[x_1, \dots, x_n]$ and $\Delta$ a pure $d$-dimensional simplicial complex on vertex set $[n]$, and 
    $$
        \F(\Delta^{(0)}), \dots, \F(\Delta^{(d - 1)})
    $$
    the facet ideals of the $0 \leq i \leq d - 1$ skeletons of $\Delta$, where $\m = \F(\Delta^{(0)})$. Then 
    $$
        e_{{\bf a}}(\m | \F(\Delta^{(1)}), \dots, \F(\Delta^{(d - 1)})) > 0
    $$
    for every ${\bf a} \in \N^d$ such that $a_0 + \dots + a_{d - 1} = n - 1$. 
\end{theorem}

\begin{proof}
    Since $\Delta$ is pure, we know by \cref{c:maxanalyticspread} that $\ell(\F(\Delta^{(i)})) = n$ for every $1 \leq i \leq d - 1$. By \cref{l:spreadproduct} we conclude for every subset $J = \{j_1, \dots, j_t\} \subset \{1, \dots, d - 1\}$ that 
    $$
        \ell(\F(\Delta^{(j_1)})\dots\F(\Delta^{(j_t)})) = n.
    $$
    In particular, the inequality from \cref{t:positivity} becomes
    $$
        a_{j_1} + \dots + a_{j_t} \leq n - 1
    $$
    which is always true since $a_0 + \dots + a_{d - 1} = n - 1$.
\end{proof}

\begin{remark}
    In the case where $\Delta$ is a simplex, \cref{t:posmm} follows directly by a classical result due to Laplace \cite{laplace}, and the equivalences between mixed multiplicities and mixed volumes from \cite{MR2320648}. This connection was first noticed in~\cite{MR3802324}, where the authors computed mixed and $j$-multiplicities of ideals generated by every squarefree monomials in a given degree in terms of (mixed) Eulerian numbers.
\end{remark}

The number of permutations of $n$ elements that have exactly $k$ descents is known as the \textbf{Eulerian number} $A(n, k)$ and it has been extensively studied, leading to several generalizations (see for example \cite{MR2487491,MR4621503,katz2023matroidal}). A classical result due to Laplace~\cite{laplace} implies the Eulerian numbers can also be described as the volume of the \textbf{hypersimplex} $\Delta_{n, k}$ which is defined as the convex hull of the following set of points
$$
    \{e_{i_1} + \dots + e_{i_k} \st \{i_1, \dots, i_k\} \subset [n]\} \subset \R^n,
$$
in other words, the convex hull of every vector with $n - k$ zero entries, and $k$ ones.
In \cite{MR2487491}, Postnikov began the study of the mixed volumes of hypersimplices $\Delta_{n, 1}, \dots, \Delta_{n, n - 1}$, which are now called the \textbf{mixed Eulerian numbers}. One of the main results of \cite{MR2487491} is a series of $9$ combinatorial properties that these numbers satisfy, the first one being the fact that every mixed Eulerian number is positive. 

Note that by taking indicator vectors of $k - 1$-faces of the $n - 1$-simplex (i.e. the face $\{1,2,3\}$ corresponds to the vector $(1,1,1,0,0,\dots, 0)$), the vertices of $\Delta_{n, k}$ can be seen as the $k - 1$-faces of a $n - 1$ dimensional simplex. From this perspective, \cref{t:posmm} is a generalization of the fact that mixed Eulerian numbers are positive. This leads us to the following definition

\begin{definition}[\textbf{Simplicial mixed Eulerian numbers}]
    Let $\Delta$ be a $d$-dimensional simplicial complex on vertex set $[n]$. Let $\F(\Delta^{(i)})$ be the facet ideal of the $i$-th skeleton of $\Delta$ for $1 \leq i < d$, and $\m = \F(\Delta^{(0)})$.

    Given ${\bf c} = (c_0, \dots, c_{d - 1}) \in \N^d$ such that $c_0 + \dots + c_{d - 1} = n - 1$, the \textbf{simplicial mixed Eulerian number} is defined as 
    $$
        A_{c_{0}, \dots, c_{d - 1}}(\Delta) := e_{{\bf c}}(\m | \F(\Delta(1)), \dots, \F(\Delta(d - 1)))
    $$
\end{definition}

\begin{remark}
    From the results in \cite{MR2320648} simplicial mixed Eulerian numbers can also be defined in terms of mixed volumes of polytopes whose vertices correspond to faces of $\Delta$.    
\end{remark}

Also in \cite{MR2487491}, Postnikov defined the notion of generalized permutohedra. Recently in~\cite[Proposition 2.6]{MR4064768}, the authors characterized generalized permutohedra as polytopes where every edge is in the direction $e_i - e_j$ for $i \neq j$. We now give an example that the polytope whose vertices correspond to the $i$-faces of a $d$-dimensional pure simplicial complex is not necessarily a generalized permutohedron. The following example was found using SageMath~\cite{sagemath}.

\begin{example}\label{e:generalizedpermutohedron}
    Let $\Delta = \tuple{\{1,2,3\}, \{1,3,4\}, \{1,4,5\}, \{1,4,5\}, \{1,5,6\}, \{1,6,7\}, \{1,2,7\}}$ and 
    
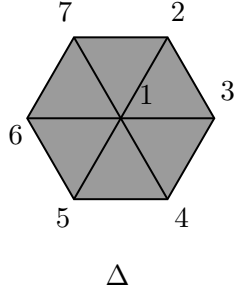
\begin{center}

    \tikzset{every picture/.style={line width=0.75pt}} 

    \begin{tikzpicture}[x=0.75pt,y=0.75pt,yscale=-1,xscale=1]
    
    \draw  [fill={rgb, 255:red, 155; green, 155; blue, 155 }  ,fill opacity=1 ] (195.33,100.17) -- (171.75,141.01) -- (124.58,141.01) -- (101,100.17) -- (124.58,59.32) -- (171.75,59.32) -- cycle ;
    \draw    (124.58,59.32) -- (148.17,100.17) ;
    \draw    (171.75,59.32) -- (148.17,100.17) ;
    \draw    (195.33,100.17) -- (148.17,100.17) ;
    \draw    (148.17,100.17) -- (171.75,141.01) ;
    \draw    (148.17,100.17) -- (124.58,141.01) ;
    \draw    (148.17,100.17) -- (101,100.17) ;
    
    \draw (139,173) node [anchor=north west][inner sep=0.75pt]   [align=left] {$\displaystyle \Delta $};
    \draw (156,82) node [anchor=north west][inner sep=0.75pt]   [align=left] {$\displaystyle 1$};
    \draw (172,40) node [anchor=north west][inner sep=0.75pt]   [align=left] {$\displaystyle 2$};
    \draw (197,80) node [anchor=north west][inner sep=0.75pt]   [align=left] {$\displaystyle 3$};
    \draw (173.75,144.01) node [anchor=north west][inner sep=0.75pt]   [align=left] {$\displaystyle 4$};
    \draw (114,144) node [anchor=north west][inner sep=0.75pt]   [align=left] {$\displaystyle 5$};
    \draw (90,102) node [anchor=north west][inner sep=0.75pt]   [align=left] {$\displaystyle 6$};
    \draw (115,40) node [anchor=north west][inner sep=0.75pt]   [align=left] {$\displaystyle 7$};

    \end{tikzpicture}
\end{center}
    $$
    E = \{e_i + e_j \in \R^7 \st \{i, j\} \text{ is an edge of $\Delta$}\}, \quad \text{ where $e_i$ denotes the standard canonical vector in $\R^7$}.
    $$
    Then the polytope $P$ that is the convex hull of $E$ is not a generalized permutohedron since 
    $$
        e_4 + e_5 = (0,0,0,1,1,0,0), \quad e_2 + e_7 = (0,1,0,0,0,0,1)
    $$
    forms an edge of $P$ and the direction of this edge is $e_4 + e_5 - e_2 - e_7$ which is not of the form $e_i - e_j$ for any $i, j$. 
\end{example}

\begin{remark}
    In~\cite{matroidpolytope}, the authors show that the convex hull of a subset of vertices of the hypersimplex is a generalized permutohedron if and only if it is the matroid basis polytope of a matroid. This characterization would also imply the polytope from~\cref{e:generalizedpermutohedron} is not a generalized permutohedron.
\end{remark}
\section{Further questions}\label{s:questions}

\cref{t:sdefectsimplex,t:triangle,t:nozero} along with \cref{e:symbolicdefectpolynomial} show that the coefficients of the second symbolic defect of $\Delta$ are sometimes equal to the $f$-vector of $\Delta$, with equality beginning at the number of triangles of $\Delta$. To see that this is not always true, take for example the graph $G$ below

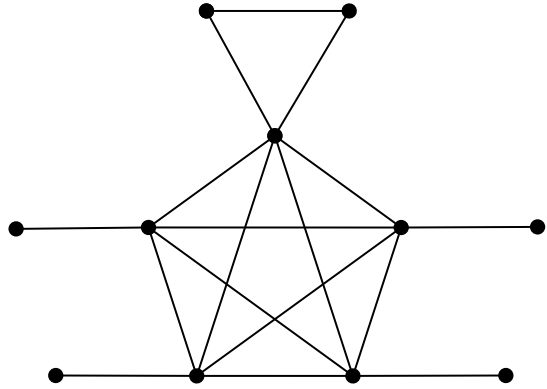
\begin{center}

\tikzset{every picture/.style={line width=0.75pt}} 

\begin{tikzpicture}[x=0.75pt,y=0.75pt,yscale=-1,xscale=1]

\draw   (348.72,174.31) -- (324.37,249.21) -- (245.61,249.2) -- (221.28,174.28) -- (285.01,128) -- cycle ;
\draw    (221.28,174.28) -- (324.37,249.21) ;
\draw    (285.01,128) -- (245.61,249.2) ;
\draw [shift={(245.61,249.2)}, rotate = 108.01] [color={rgb, 255:red, 0; green, 0; blue, 0 }  ][fill={rgb, 255:red, 0; green, 0; blue, 0 }  ][line width=0.75]      (0, 0) circle [x radius= 3.35, y radius= 3.35]   ;
\draw [shift={(285.01,128)}, rotate = 108.01] [color={rgb, 255:red, 0; green, 0; blue, 0 }  ][fill={rgb, 255:red, 0; green, 0; blue, 0 }  ][line width=0.75]      (0, 0) circle [x radius= 3.35, y radius= 3.35]   ;
\draw    (285.01,128) -- (324.37,249.21) ;
\draw [shift={(324.37,249.21)}, rotate = 72.01] [color={rgb, 255:red, 0; green, 0; blue, 0 }  ][fill={rgb, 255:red, 0; green, 0; blue, 0 }  ][line width=0.75]      (0, 0) circle [x radius= 3.35, y radius= 3.35]   ;
\draw [shift={(285.01,128)}, rotate = 72.01] [color={rgb, 255:red, 0; green, 0; blue, 0 }  ][fill={rgb, 255:red, 0; green, 0; blue, 0 }  ][line width=0.75]      (0, 0) circle [x radius= 3.35, y radius= 3.35]   ;
\draw    (348.72,174.31) -- (245.61,249.2) ;
\draw    (221.28,174.28) -- (348.72,174.31) ;
\draw [shift={(348.72,174.31)}, rotate = 0.01] [color={rgb, 255:red, 0; green, 0; blue, 0 }  ][fill={rgb, 255:red, 0; green, 0; blue, 0 }  ][line width=0.75]      (0, 0) circle [x radius= 3.35, y radius= 3.35]   ;
\draw [shift={(221.28,174.28)}, rotate = 0.01] [color={rgb, 255:red, 0; green, 0; blue, 0 }  ][fill={rgb, 255:red, 0; green, 0; blue, 0 }  ][line width=0.75]      (0, 0) circle [x radius= 3.35, y radius= 3.35]   ;
\draw    (250.5,65) -- (285.01,128) ;
\draw [shift={(250.5,65)}, rotate = 61.29] [color={rgb, 255:red, 0; green, 0; blue, 0 }  ][fill={rgb, 255:red, 0; green, 0; blue, 0 }  ][line width=0.75]      (0, 0) circle [x radius= 3.35, y radius= 3.35]   ;
\draw    (322.5,65) -- (285.01,128) ;
\draw [shift={(322.5,65)}, rotate = 120.75] [color={rgb, 255:red, 0; green, 0; blue, 0 }  ][fill={rgb, 255:red, 0; green, 0; blue, 0 }  ][line width=0.75]      (0, 0) circle [x radius= 3.35, y radius= 3.35]   ;
\draw    (250.5,65) -- (322.5,65) ;
\draw [shift={(250.5,65)}, rotate = 0] [color={rgb, 255:red, 0; green, 0; blue, 0 }  ][fill={rgb, 255:red, 0; green, 0; blue, 0 }  ][line width=0.75]      (0, 0) circle [x radius= 3.35, y radius= 3.35]   ;
\draw    (417.5,174) -- (348.72,174.31) ;
\draw [shift={(417.5,174)}, rotate = 179.74] [color={rgb, 255:red, 0; green, 0; blue, 0 }  ][fill={rgb, 255:red, 0; green, 0; blue, 0 }  ][line width=0.75]      (0, 0) circle [x radius= 3.35, y radius= 3.35]   ;
\draw    (401.5,249) -- (324.37,249.21) ;
\draw [shift={(401.5,249)}, rotate = 179.84] [color={rgb, 255:red, 0; green, 0; blue, 0 }  ][fill={rgb, 255:red, 0; green, 0; blue, 0 }  ][line width=0.75]      (0, 0) circle [x radius= 3.35, y radius= 3.35]   ;
\draw    (174.5,249) -- (245.61,249.2) ;
\draw [shift={(174.5,249)}, rotate = 0.16] [color={rgb, 255:red, 0; green, 0; blue, 0 }  ][fill={rgb, 255:red, 0; green, 0; blue, 0 }  ][line width=0.75]      (0, 0) circle [x radius= 3.35, y radius= 3.35]   ;
\draw    (154.5,175) -- (221.28,174.28) ;
\draw [shift={(154.5,175)}, rotate = 359.39] [color={rgb, 255:red, 0; green, 0; blue, 0 }  ][fill={rgb, 255:red, 0; green, 0; blue, 0 }  ][line width=0.75]      (0, 0) circle [x radius= 3.35, y radius= 3.35]   ;

\end{tikzpicture}
\end{center}

The second symbolic defect polynomial of $\Ind(G)$ is given by $\mu_2(\Delta, t) = 58t^3 + 57t^4 + 60 t^5$, while its $f$-vector is given by $(1, 11, 38, 58, 41, 11)$. 

A natural question that arises is when does equality hold. Based on computational evidence and the examples above we ask the following question:

\begin{question}\label{q:forest}
    Let $G$ be a forest, $\Ind(G)$ its independence complex, and $(f_{-1}, f_0, f_1, \dots, f_d)$ the $f$-vector of $\Ind(G)$. Does the following equality always hold?
    $$
        \mu_2(\Ind(G), t) = f_2 t^3 + \dots + f_{d} t^{d + 1}
    $$
    In other words, if $f(G, t)$ is the independence polynomial of a forest $G$, does the following always hold?
    $$
        f(G, t) = 1 + |V(G)|t + \Big{(}\binom{n}{2} - |E(G)|\Big{)}t^2 + \mu_2(\Delta, t) 
    $$
\end{question}
 
A sequence $a_1, \dots, a_d$ of integers is said to be \textbf{unimodal} if there exists $1 \leq k \leq d$ such that $a_1 \leq \dots \leq a_k \geq \dots \geq a_d$. Although the previous example shows the sequence of coefficients of the second symbolic defect polynomial of a flag complex is not always unimodal, extensive computations lead us to the following question. 

\begin{question}\label{q:unimodalforest}
    When do the coefficients of the $m$-th symbolic defect polynomial of a complex $\Delta$ form a unimodal sequence with no internal zeros?

    More specifically, do the coefficients of the second symbolic defect polynomial of the independence complex of a forest form a unimodal sequence with no internal zeros?
\end{question}

Note that a positive answer to \cref{q:forest} and the second part of \cref{q:unimodalforest} together imply the conjecture by Alavi, Malde, Schwenk and Erd\"os~\cite{MR0944684} on the unimodality of the coefficients of the independence polynomial of forests.

We say a sequence of polynomials $\{\mu_m(\Delta, t)\}_m$ is \textbf{asymptotically quasi-unimodal} if there exists a $k$ such that the coefficients of $\mu_l(\Delta, t)$ form a unimodal sequence for every $l \gg 0$ such that $k \mid l$. In view of the results from \cite{MR3986423,oltsik} we ask the following question.

\begin{question}
    Let $\Delta$ be a simplicial complex. Is the sequence of symbolic defect polynomials $\{\mu_m(\Delta, t)\}_m$ asymptotically quasi-unimodal?
\end{question}

In \cref{t:forestslp} we showed $2$-dimensional simplicial trees satisfy the SLP in every odd characteristic. \cref{p:incidenceforest,hausel} can be used to show that many multiplication maps by a general linear form have full rank for arbitrary pure simplicial forests. We then ask the following.

\begin{question}
    Let $\Delta$ be a pure simplicial forest. Does $A(\Delta)$ satisfy the SLP in characteristic zero?
\end{question}
 
\cref{t:posmm}, the results in \cite{MR2487491} and the discussion in \cref{s:mvmm} lead us to the following question:

\begin{question}
    Let $\Delta$ be a pure $d$-dimensional simplicial complex. 
    
    \begin{enumerate}
        \item Is there a combinatorial interpretation to the sum
        $$
            \sum A_{c_0,\dots, c_{d - 1}}(\Delta) ?
        $$
        \item Let $A(i, \Delta) = A_{c_0,\dots, c_{d - 1}}(\Delta)$ where $c_i = n - 1$ and $c_j = 0$ for every $j \neq i$. Is the sequence
        $$
            A(0, \Delta), \dots, A(d - 1, \Delta)
        $$
        unimodal?
    \end{enumerate} 
\end{question}

\noindent
{\bf Acknowledgments.} I would like to thank Zaqueu Ramos for insightful conversations on the linear type property; Arvind Kumar, Susan Cooper and Susan Morey for pointing me to references on symbolic powers and my supervisor Sara Faridi for very useful discussions which improved the presentation of the paper. Computations for this project were made using the NautyGraphs~\cite{NautyGraphsSource}, SymbolicPowers~\cite{SymbolicPowersSource,SymbolicPowersArticle}, ReesAlgebra~\cite{ReesAlgebraSource,ReesAlgebraArticle} and SimplicialComplexes~\cite{SimplicialComplexesSource} packages from Macaulay2 \cite{M2}, and SageMath~\cite{sagemath}.

\bibliographystyle{plain}


\end{document}